\documentclass[a4paper, 10pt]{article}

\usepackage{packageList}
\usepackage{macros}

\graphicspath{{./Figures/}}
\pgfplotsset{compat=1.11}
\newlength\fwidth

\title{On the optimal shape parameter for kernel methods:
Sharp direct and inverse statements} 
\author[1,2]{Tizian Wenzel \thanks{\href{mailto:wenzel@math.lmu.de}{wenzel@math.lmu.de}}}
\affil[1]{Ludwig Maximilian University of Munich (Munich, Germany)}
\affil[2]{Munich Center for Machine Learning (Munich, Germany)}

\author[3]{Gabriele Santin\thanks{\href{mailto:gabriele.santin@unive.it}{gabriele.santin@unive.it}%
}} 
\affil[3]{Department of Environmental Sciences, Informatics and Statistics, Ca' Foscari University of Venice (Venice, Italy)}

\newif\iflong			
\longfalse

\begin{document}

\maketitle %

\begin{abstract}
The search for the optimal shape parameter for Radial Basis Function (RBF) kernel approximation has been an outstanding research problem for decades.
In this work, 
we establish a theoretical framework for this problem by leveraging a recently established theory on sharp direct, inverse and saturation statements for kernel based approximation.
In particular, we link the search for the optimal shape parameter to superconvergence phenomena.

Our analysis is carried out for finitely smooth Sobolev kernels, thereby covering large classes of radial kernels used in practice, including those emerging from current machine-learning methodologies.

Our results elucidate how approximation regimes, kernel regularity, and parameter choices interact, thereby clarifying a question that has remained unresolved for decades.
\end{abstract}

~ \newpage

\section{Introduction}
\label{sec:introduction}

Kernel methods constitute a cornerstone of modern numerical analysis
and data-driven modeling, 
with deep roots in approximation theory and statistical learning theory and with growing influence in contemporary machine learning and deep learning \cite{wendland2005scattered,steinwart2008support,fasshauer2015book}. 
Their success across such a broad range of disciplines crucially relies on the ability to encode prior structural assumptions through a suitable choice of 
the kernel. 
Consequently, the selection of an appropriate kernel is a recurring and practically relevant task, 
affecting both theoretical guarantees and empirical performances.
Within this landscape, 
Radial Basis Function (RBF) kernels form a particularly important and widely used class, 
owing to their conceptual simplicity, 
ease of implementation, and well-understood analytical properties. 
For these kernels, the problem of choosing a ``good'' kernel often reduces to selecting an optimal shape parameter, 
which governs the spatial scale or effective width of the kernel and has a decisive impact on approximation quality and stability. 
In this article, we address this issue in the setting of interpolation with finitely smooth kernels,
while discussing how the proposed theory extends to further kernel-based methods.

To be precise, we consider interpolation of continuous functions with strictly positive definite kernels $k: \Omega \times \Omega \rightarrow \R$ on some bounded Lipschitz region $\Omega \subset \R^d$.
This means that,
for any continuous function $f \in \mathcal{C}(\Omega)$ and a set of points $X \subset \Omega$,
we are interested in a kernel interpolant $s_{f, k, X}$ satisfying the interpolation conditions
\begin{align}
\label{eq:interpol_conditions}
s_{f, k, X}(x_i) = f(x_i) \qquad \forall x_i \in X.
\end{align}
For any strictly positive definite kernel $k$,
there exists a unique reproducing kernel Hilbert space $\ns$ (RKHS) of functions,
which is also called \textit{native space}.
The minimum-norm interpolant $s_{f, k, X}$ satisfying Eq.~\eqref{eq:interpol_conditions} can be expressed as
\begin{align*}
s_{f, k, X} = \sum_{j=1}^{|X|} \alpha_j k(\cdot, x_j),
\end{align*}
with $X$-dependent coefficients $(\alpha_j)_{j=1}^{|X|}$ that can be determined by the unique solution of the linear equation system arising from Eq.~\eqref{eq:interpol_conditions}.

An important class of kernels is given by translational invariant ones,
which can be written as $k(x, z) = \Phi(x-z)$ for some function $\Phi: \R^d \rightarrow \R$.
The positive definiteness of such kernels can be analyzed with help of the Fourier transform $\hat{\Phi}$.
We will be interested in finitely smooth kernels, i.e.\ it holds
\begin{align}
\label{eq:fourier_decay}
c_\Phi (1 + \Vert \omega \Vert_2^2)^{-\tau} \leq \hat{\Phi}(\omega) \leq C_\Phi (1 + \Vert \omega \Vert_2^2)^{-\tau}
\end{align}
for constants $0 < c_\Phi < C_\Phi < \infty$ and some $\tau > d/2$.
A particular subclass of translational invariant kernels is given by RBF kernels,
that can be expressed as
\begin{align}
\label{eq:rbf_kernel}
k_\varepsilon(x, z) = \varphi(\varepsilon \Vert x - z \Vert)
\end{align}
for some scalar valued shape parameter $\varepsilon > 0$ and a univariate function $\varphi: \R \rightarrow \R$.
Due to their simple implementation, these kernels are popular in applications.
Well-known examples of RBF kernels are e.g.\ the family of Matérn or Wendland kernels,
or the Gaussian kernel \cite{buhmann2003radial,wendland2005scattered}.

The choice of the shape parameter $\varepsilon > 0$ in Eq.~\eqref{eq:rbf_kernel} is a very important and frequently studied topic in kernel based approximation and applications~\cite{larsson2024scaling,sun2026optimizing}.
The reason for this importance is that this parameter crucially influences the accuracy, stability,
complexity (computational efficiency) and thus scalability of the kernel interpolant $s_{f, k, X}$.
Therefore, the optimization of $\varepsilon$ has been the topic of extensive research over three decades \cite{sun2026optimizing},
with important practical and algorithmical contributions \cite{rippa1999algorithm,fasshauer2007choosing}.
We refer to \Cref{fig:intro_placeholder} for an exemplary visualization of the $L_2(\Omega)$ error of the kernel interpolant $s_{f, k, X}$ over the shape parameter $\varepsilon$ when using the kernel $k_\varepsilon$ satisfying Eq.~\eqref{eq:rbf_kernel}.

\begin{figure}[htbp]
  \centering
\setlength\fwidth{.48\textwidth}
\begin{tikzpicture}

\definecolor{darkgray176}{RGB}{176,176,176}
\definecolor{lightgray204}{RGB}{204,204,204}

\begin{axis}[
legend cell align={left},
legend style={
  fill opacity=0.8,
  draw opacity=1,
  text opacity=1,
  at={(0.03,0.97)},
  anchor=north west,
  draw=lightgray204
},
log basis x={10},
log basis y={10},
tick align=outside,
tick pos=left,
x grid style={darkgray176},
xlabel={shape parameter $\varepsilon$},
xmajorgrids,
xmin=0.000562341325190349, xmax=177.827941003892,
xmode=log,
xtick style={color=black},
y grid style={darkgray176},
ylabel={$L_2(\Omega)$ error},
ymajorgrids,
ymin=1.36277865224791e-05, ymax=0.0159970038739627,
ymode=log,
ytick style={color=black}
]
\addplot [very thick, red]
table {%
0.001 3.85979037119366e-05
0.00112201845430196 3.85941812617023e-05
0.00125892541179417 3.85900047379448e-05
0.00141253754462275 3.85853187699836e-05
0.00158489319246111 3.85800612411563e-05
0.00177827941003892 3.85741624605391e-05
0.00199526231496888 3.85675442556846e-05
0.00223872113856834 3.85601189333068e-05
0.00251188643150958 3.85517881163411e-05
0.00281838293126445 3.85424414576961e-05
0.00316227766016838 3.85319551798239e-05
0.00354813389233575 3.8520190448604e-05
0.00398107170553497 3.85069915469192e-05
0.00446683592150963 3.84921838274309e-05
0.00501187233627272 3.8475571427241e-05
0.00562341325190349 3.84569346957778e-05
0.00630957344480193 3.84360273289948e-05
0.00707945784384138 3.84125731539153e-05
0.00794328234724281 3.83862625274044e-05
0.00891250938133746 3.83567483204019e-05
0.01 3.83236414168489e-05
0.0112201845430196 3.82865056916377e-05
0.0125892541179417 3.82448523980117e-05
0.0141253754462275 3.81981339050475e-05
0.0158489319246111 3.81457367159843e-05
0.0177827941003892 3.80869736870392e-05
0.0199526231496888 3.80210753692071e-05
0.0223872113856834 3.79471803835268e-05
0.0251188643150958 3.78643247394309e-05
0.0281838293126445 3.77714300004743e-05
0.0316227766016838 3.76672901996926e-05
0.0354813389233576 3.75505574141158e-05
0.0398107170553497 3.74197259079456e-05
0.0446683592150963 3.72731147724272e-05
0.0501187233627272 3.71088490165423e-05
0.0562341325190349 3.69248390982444e-05
0.0630957344480193 3.67187589537656e-05
0.0707945784384138 3.64880226745563e-05
0.0794328234724282 3.6229760122815e-05
0.0891250938133746 3.5940791989113e-05
0.1 3.56176051028463e-05
0.112201845430196 3.52563292677494e-05
0.125892541179417 3.48527175731793e-05
0.141253754462275 3.44021331488727e-05
0.158489319246111 3.38995468588412e-05
0.177827941003892 3.33395527521358e-05
0.199526231496888 3.27164116583721e-05
0.223872113856834 3.20241388753567e-05
0.251188643150958 3.12566606757347e-05
0.281838293126445 3.04080784280543e-05
0.316227766016838 2.94731019955582e-05
0.354813389233576 2.84477517163072e-05
0.398107170553497 2.73304907548035e-05
0.446683592150963 2.6124053289702e-05
0.501187233627272 2.48384026877375e-05
0.562341325190349 2.34955114937071e-05
0.630957344480194 2.21369815219417e-05
0.707945784384138 2.08356828426409e-05
0.794328234724282 1.9711643901068e-05
0.891250938133746 1.89479048098374e-05
1 1.87911431214358e-05
1.12201845430196 1.95125512655718e-05
1.25892541179417 2.1328719550001e-05
1.41253754462276 2.43421402894671e-05
1.58489319246111 2.85576760518968e-05
1.77827941003892 3.39450535604523e-05
1.99526231496888 4.04882591246678e-05
2.23872113856834 4.82056200609016e-05
2.51188643150958 5.71527104240085e-05
2.81838293126446 6.74200285731889e-05
3.16227766016838 7.91308840985016e-05
3.54813389233576 9.24411389651859e-05
3.98107170553497 0.0001075411191175
4.46683592150963 0.000124659766032
5.01187233627272 0.0001440712160467
5.62341325190349 0.000166104226427
6.30957344480194 0.0001911551469538
7.07945784384138 0.0002197054613205
7.94328234724282 0.0002523453058367
8.91250938133746 0.0002898048244816
10 0.0003329957407947
11.2201845430196 0.0003830661052082
12.5892541179417 0.000441471779007
14.1253754462276 0.0005100687921527
15.8489319246111 0.0005912312011866
17.7827941003892 0.0006879994276634
19.9526231496888 0.0008042642748803
22.3872113856834 0.0009449919706815
25.1188643150958 0.0011164957955647
28.1838293126446 0.0013267602288907
31.6227766016838 0.0015858239860261
35.4813389233575 0.0019062283333603
39.8107170553498 0.0023035356264028
44.6683592150963 0.0027969185357256
50.1187233627273 0.0034098108581104
56.2341325190349 0.0041705937479732
63.0957344480194 0.0051132639816516
70.7945784384139 0.00627799082669
79.4328234724282 0.0077114133868586
89.1250938133746 0.0094664617630712
100 0.0116014098974612
};
\addlegendentry{{ \scriptsize $|X| = 100$}}
\addplot [very thick, red, dashed, forget plot]
table {%
0.001 3.85979037119366e-05
0.00112201845430196 3.85979037119366e-05
0.00125892541179417 3.85979037119366e-05
0.00141253754462275 3.85979037119366e-05
0.00158489319246111 3.85979037119366e-05
0.00177827941003892 3.85979037119366e-05
0.00199526231496888 3.85979037119366e-05
0.00223872113856834 3.85979037119366e-05
0.00251188643150958 3.85979037119366e-05
0.00281838293126445 3.85979037119366e-05
0.00316227766016838 3.85979037119366e-05
0.00354813389233575 3.85979037119366e-05
0.00398107170553497 3.85979037119366e-05
0.00446683592150963 3.85979037119366e-05
0.00501187233627272 3.85979037119366e-05
0.00562341325190349 3.85979037119366e-05
0.00630957344480193 3.85979037119366e-05
0.00707945784384138 3.85979037119366e-05
0.00794328234724281 3.85979037119366e-05
0.00891250938133746 3.85979037119366e-05
0.01 3.85979037119366e-05
0.0112201845430196 3.85979037119366e-05
0.0125892541179417 3.85979037119366e-05
0.0141253754462275 3.85979037119366e-05
0.0158489319246111 3.85979037119366e-05
0.0177827941003892 3.85979037119366e-05
0.0199526231496888 3.85979037119366e-05
0.0223872113856834 3.85979037119366e-05
0.0251188643150958 3.85979037119366e-05
0.0281838293126445 3.85979037119366e-05
0.0316227766016838 3.85979037119366e-05
0.0354813389233576 3.85979037119366e-05
0.0398107170553497 3.85979037119366e-05
0.0446683592150963 3.85979037119366e-05
0.0501187233627272 3.85979037119366e-05
0.0562341325190349 3.85979037119366e-05
0.0630957344480193 3.85979037119366e-05
0.0707945784384138 3.85979037119366e-05
0.0794328234724282 3.85979037119366e-05
0.0891250938133746 3.85979037119366e-05
0.1 3.85979037119366e-05
0.112201845430196 3.85979037119366e-05
0.125892541179417 3.85979037119366e-05
0.141253754462275 3.85979037119366e-05
0.158489319246111 3.85979037119366e-05
0.177827941003892 3.85979037119366e-05
0.199526231496888 3.85979037119366e-05
0.223872113856834 3.85979037119366e-05
0.251188643150958 3.85979037119366e-05
0.281838293126445 3.85979037119366e-05
0.316227766016838 3.85979037119366e-05
0.354813389233576 3.85979037119366e-05
0.398107170553497 3.85979037119366e-05
0.446683592150963 3.85979037119366e-05
0.501187233627272 3.85979037119366e-05
0.562341325190349 3.85979037119366e-05
0.630957344480194 3.85979037119366e-05
0.707945784384138 3.85979037119366e-05
0.794328234724282 3.85979037119366e-05
0.891250938133746 3.85979037119366e-05
1 3.85979037119366e-05
1.12201845430196 3.85979037119366e-05
1.25892541179417 3.85979037119366e-05
1.41253754462276 3.85979037119366e-05
1.58489319246111 3.85979037119366e-05
1.77827941003892 3.85979037119366e-05
1.99526231496888 3.85979037119366e-05
2.23872113856834 3.85979037119366e-05
2.51188643150958 3.85979037119366e-05
2.81838293126446 3.85979037119366e-05
3.16227766016838 3.85979037119366e-05
3.54813389233576 3.85979037119366e-05
3.98107170553497 3.85979037119366e-05
4.46683592150963 3.85979037119366e-05
5.01187233627272 3.85979037119366e-05
5.62341325190349 3.85979037119366e-05
6.30957344480194 3.85979037119366e-05
7.07945784384138 3.85979037119366e-05
7.94328234724282 3.85979037119366e-05
8.91250938133746 3.85979037119366e-05
10 3.85979037119366e-05
11.2201845430196 3.85979037119366e-05
12.5892541179417 3.85979037119366e-05
14.1253754462276 3.85979037119366e-05
15.8489319246111 3.85979037119366e-05
17.7827941003892 3.85979037119366e-05
19.9526231496888 3.85979037119366e-05
22.3872113856834 3.85979037119366e-05
25.1188643150958 3.85979037119366e-05
28.1838293126446 3.85979037119366e-05
31.6227766016838 3.85979037119366e-05
35.4813389233575 3.85979037119366e-05
39.8107170553498 3.85979037119366e-05
44.6683592150963 3.85979037119366e-05
50.1187233627273 3.85979037119366e-05
56.2341325190349 3.85979037119366e-05
63.0957344480194 3.85979037119366e-05
70.7945784384139 3.85979037119366e-05
79.4328234724282 3.85979037119366e-05
89.1250938133746 3.85979037119366e-05
100 3.85979037119366e-05
};
\end{axis}

\end{tikzpicture}
\caption{Visualization of the $\Vert f - s_{f, k, X} \Vert_{L_2(\Omega)}$ error over the shape parameter $\varepsilon$.
This is a typical situation frequently reported in the literature,
which will be explained by our analysis.
In fact, the red curve is depicted again in the middle left plot of \Cref{fig:ex2_1d}.}
  \label{fig:intro_placeholder}
\end{figure}
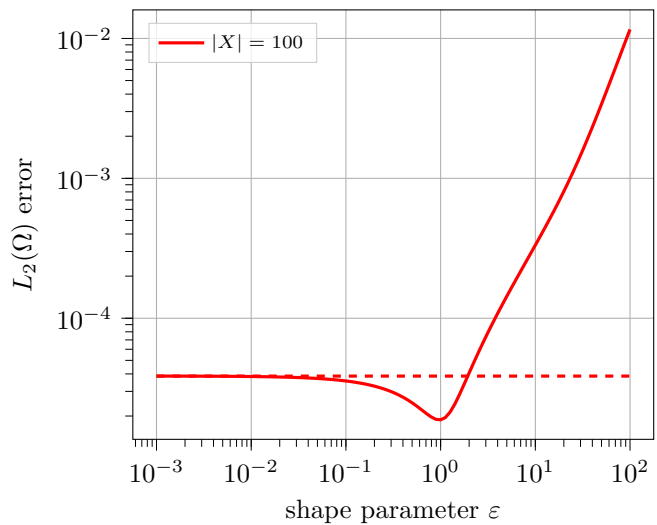

Despite its practical importance, 
only little theoretical progress has been obtained for the search of the optimal shape parameter -- even a precise definition of an optimal shape parameter is missing.
One reason for the lack of theory (with respect to accuracy) is that the shape parameter $\varepsilon > 0$ is not explicitly present in standard error estimates 
using the fill distance $h_X$, such as
\begin{align}
\label{eq:standard_bound}
\begin{aligned}
\Vert f - s_{f, k, X} \Vert_{L_2(\Omega)} &\leq C h_X^\tau \Vert f - s_{f, k, X} \Vert_{\ns} \\
&\leq C h_X^\tau \Vert f \Vert_{\ns},
\end{aligned}
\end{align}
which hold for $f \in \ns$ and kernels satisfying Eq.~\eqref{eq:fourier_decay}.
Eq.~\eqref{eq:standard_bound} suggests that the shape parameter does not significantly influence the error,
which is however apparently not true in view of extensive numerical evidence,
cf.\ \Cref{fig:intro_placeholder}.

In this article,
we address the search for an optimal shape parameter by providing a solid theoretical analysis.
This will be achieved by combining recent progress in two different areas of kernel based approximation, 
namely kernel-based superconvergence \cite{schaback1995error,schaback2018superconvergence,sloan2025doubling,karvonen2025general}
as well as sharp inverse and saturation statements for kernel based approximation 
\cite{wendland1999meshless,wenzel2025sharp,avesani2025sobolev,wenzel2026sharp}.
All together, these works established a theory of sharp direct and inverse statements and therefore a one-to-one correspondence between smoothness of a 
function and rate of approximation.
Thus, by leveraging these theoretical result we precisely characterize the notion of an \emph{asymptotically optimal} shape parameter,
and providing several implications and applications that explain e.g.\ \Cref{fig:intro_placeholder}.
In particular, our analysis is not restricted to scalar valued shape parameter as in Eq.~\eqref{eq:rbf_kernel},
but naturally holds for a broad class of kernel parametrizations.

The article is structured as follows: 
In \Cref{sec:background} we review necessary background information on kernels, 
sharp direct and inverse statements as well as literature on the optimal shape parameter.
\Cref{sec:optimal_shape_para} establishes the main result in \Cref{th:main_result},
and \Cref{sec:implications_applications} discusses the resulting implications and applications.
\Cref{sec:conclusion_outlook} concludes the article.

\section{Background}
\label{sec:background}

We start by reviewing some material on kernels (\Cref{subsec:kernels_rkhs_power}), recall recent results on sharp direct and inverse statements (\Cref{subsec:sharp_direct_inverse}), and discuss some related work (\Cref{subsec:literature_optimal_shape_para}).

\subsection{Kernels, RKHS and power spaces}
\label{subsec:kernels_rkhs_power}

\subsubsection{Kernels and RKHS}

A strictly positive definite kernel $k: \Omega \times \Omega \rightarrow \R$ is a symmetric function such that the kernel matrix $A_X := (k(x_i, x_j))_{1 \leq 
i, j \leq |X|} \in \R^{|X| \times |X|}$ 
is positive definite for any choice of pairwise disjoint points $X \subset \Omega$.
We will consider the case of $\Omega \subset \R^d$ being a Lipschitz region.
For such kernels,
there is a unique reproducing kernel Hilbert space $\ns$ (RKHS) of functions,
such that it holds
\begin{align}
\begin{aligned}
\label{eq:reproducing_property}
&k(\cdot, x) \in \ns \qquad ~ \qquad \forall x \in \Omega, \\
&f(x) = \langle f, k(\cdot, x) \rangle_{\ns} \quad \forall x \in \Omega, ~ \forall f \in \ns \hfill.
\end{aligned}
\end{align}
Assuming that the kernel $k$ is a continuous translational invariant kernel satisfying Eq.~\eqref{eq:fourier_decay} for some $\tau > d/2$, 
one can show that the RKHS is norm-equivalent to the Sobolev space of smoothness $\tau$,
i.e.\ $\ns \asymp H^\tau(\Omega)$.
This means that the spaces are equal as sets, 
and that their norms are equivalent.
However, for the subsequent analysis
we will not focus on translational invariant kernels satisfying Eq.~\eqref{eq:fourier_decay},
but directly assume that the kernel has an RKHS that is norm-equivalent to a Sobolev space.
We will call such kernels \textit{Sobolev kernels}:
\begin{assumption}
\label{ass:kernel_domain}
Let $\Omega \subset \R^d$ be a compact Lipschitz region and let $k: \Omega \times \Omega \rightarrow \R$ be a continuous kernel such that $\ns \asymp 
H^\tau(\Omega)$ for some $\tau > d/2$. 
\end{assumption}

If Eq.~\eqref{eq:fourier_decay} holds for any $\varepsilon > 0$,
it is well known \cite{larsson2024scaling} that the norm equivalence $\mathcal{H}_{k_\varepsilon}(\Omega) \asymp H^\tau(\Omega)$ holds for all choices of the shape parameter $\varepsilon > 0$ in Eq.~\eqref{eq:rbf_kernel}
-- this follows essentially from the dilation property of the Fourier transform.
This observation will be used in \Cref{ass:kernel_shape} to consider general parametrized families of Sobolev kernels.

\subsubsection{Geometric quantities}

In order to quantify the approximation error between the kernel interpolant $s_{f, k, X}$ and a target function $f \in \mathcal{C}(\Omega)$ of interest,
one considers the fill distance $h_X$ as well as the separation distance $q_X$:
\begin{align}
\begin{aligned}
\label{eq:fill_sep_dist}
h_X &:= h_{X, \Omega} := \sup_{x \in \Omega} \min_{x_j \in X} \Vert x - x_j \Vert_2, \\
q_X &:= \frac{1}{2} \min_{x_i \neq x_j \in X} \Vert x_i - x_j \Vert_2.
\end{aligned}
\end{align}
The fill and seperation distances defined in Eq.~\eqref{eq:fill_sep_dist} allow one to quantity the error between the kernel interpolant $s_{f, k, X}$ and a target
function $f \in \mathcal{C}(\Omega)$ of interest.
Such results will be reviewed in the subsequent \Cref{subsec:sharp_direct_inverse}.

In order to assess the uniformity of a point set $X \subset \Omega$,
one considers the so-called uniformity constant $\rho_X$ defined as
\begin{align}
\label{eq:uniformity_constant}
\rho_X := \frac{h_X}{q_X}.
\end{align}
A sequence $(X_n)_{n \in \N} \subset \Omega$ of point sets is called quasi-uniform
if the uniformity constants $\rho_{X_n}$ are uniformly bounded with respect to $n$.
The notion of quasi-uniform point sets is frequently considered in the literature,
and also relevant for the following analysis.
Thus we state the following assumption:

\begin{assumption}
\label{ass:points}
Let $(X_n)_{n \in \N} \subset \Omega$ be a nested sequence of quasi-uniformly distributed point sets with geometrically decaying fill distance,
s.t.\ it holds
\begin{align}
\label{eq:assumption_decay_fill_dist}
c_0' a^n \leq q_{X_n} \leq h_{X_n, \Omega} \leq c_0'' a^n
\end{align}
for constants $c_0', c_0'' > 0$ and $a \in (0, 1)$.
\end{assumption}

\subsubsection{Power spaces}

Particular super- and subspaces of the RKHS $\ns$ are of crucial importance for the subsequent theoretical analysis.
By Mercer theorem, for any positive definite kernel $k: \Omega \times \Omega \rightarrow \R$ on a compact Lipschitz region $\Omega \subset \R^d$
the integral operator $T: L_2(\Omega) \rightarrow L_2(\Omega)$ defined as
\begin{align*}
(Tf)(x) := (T_k f)(x) := \int_\Omega k(x, z) f(z) ~ \mathrm{d}z
\end{align*}
is a positive and compact operator. Its ordered eigenvalues $(\lambda_j)_{j \in \N} \subset \R_+$ and corresponding eigenfunctions $(\varphi_j)_{j \in \N} \subset L_2(\Omega)$ that
form an orthonormal basis of $L_2(\Omega)$ and are orthogonal in the RKHS $\ns$.
Using these eigenvalues and eigenfunctions,
the power spaces $(\ns)_{\vartheta \geq 0}$ \cite{steinwart2012mercer} are given as
\begin{align}
\label{eq:power_space}
(\mathcal{H}_{k}(\Omega))_\vartheta := \left\{ f \in L_2(\Omega) ~:~ \sum_{j=1}^\infty \frac{|\langle f, \varphi_{j} 
\rangle_{L_2(\Omega)}|^2}{\lambda_{j}^\vartheta} < \infty \right\} \subset L_2(\Omega).
\end{align}
The inclusions $(\mathcal{H}_{k}(\Omega))_{\vartheta_1} \hookrightarrow (\mathcal{H}_{k}(\Omega))_{\vartheta_2}$ for $\vartheta_1 > \vartheta_2$ are obvious,
and it holds
$(\mathcal{H}_{k}(\Omega))_0 = L_2(\Omega)$,
$(\mathcal{H}_{k}(\Omega))_1 = \ns$ and $(\mathcal{H}_{k}(\Omega))_2 = TL_2(\Omega)$.
We will only consider the power spaces for $\vartheta \geq 0$,
thus dealing with subspaces of $L_2(\Omega)$.
A visualization of these power spaces is given in \Cref{fig:visualization}.

Under \Cref{ass:kernel_domain}, 
i.e.\ $\ns \asymp H^\tau(\Omega)$,
we obtain by real interpolation between $L_2(\Omega)$ and $\ns \asymp H^\tau(\Omega)$ that
\begin{align}
\label{eq:power_spaces_escaping}
(\mathcal{H}_{k}(\Omega))_\vartheta \asymp H^{\vartheta \tau}(\Omega) \qquad \text{for } \vartheta \in [0, 1],
\end{align}
i.e.\ the power spaces essentially only differ for $\vartheta > 1$,
where one has the subset relation
\begin{align}
\label{eq:power_spaces_superconv}
(\mathcal{H}_k(\Omega))_\vartheta \subset H^{\vartheta \tau}(\Omega) \qquad \text{for } \vartheta \in [1, \infty).
\end{align}
As elaborated in \cite{karvonen2025general},
the spaces $(\mathcal{H}_k(\Omega))_\vartheta$ for $\vartheta > 1$ are expected to be characterized by the Sobolev smoothness $\vartheta \tau$ and possibly 
additional \emph{boundary} conditions,
though this is not yet fully understood.

\subsection{Sharp direct and inverse statements}
\label{subsec:sharp_direct_inverse}

We recall the following direct statement,
which is a summary of a long history of work which covered both the escaping the native space regime 
\cite{narcowich2005sobolev,wendland2005approximate,narcowich2006sobolev,arcangeli2007extension} (abbreviated as \textit{escaping regime} in the following)
as well as the superconvergence regime \cite{schaback1999improved,schaback2018superconvergence,sloan2025doubling,karvonen2025general}.

\begin{theorem}
\label{th:result_direct}
Under \Cref{ass:kernel_domain},
consider $f \in \mathcal{C}(\Omega)$ such that $f \in (\ns)_\vartheta$ for some $\vartheta \in (\frac{d}{2\tau}, 2]$ 
Let $\rho_0 > 0$.
Then there exists a constant $c_f > 0$ such that it holds
\begin{align*}
\Vert f - s_{f, k, X} \Vert_{L_2(\Omega)} \leq c_f
h_X^{\tau \min(\vartheta, 2)}
\end{align*}
for any set of quasi-uniform points $X \subset \Omega$ with $\frac{h_X}{q_X} \leq \rho_0$ and $h_X\leq 1$.
\end{theorem}

\begin{proof}
This is a combination of the direct statements from the escaping regime $\vartheta \in (\frac{d}{2\tau}, 1]$ from \cite[Theorem 4.2]{narcowich2006sobolev} and \cite[Theorem 
4.1]{arcangeli2007extension},
and the superconvergence regime $\vartheta \in [1, 2]$ \cite{karvonen2025general}:

In the superconvergence regime, 
due to the nestedness of the power spaces,
it holds $(\ns)_\vartheta \hookrightarrow (\ns)_2$
for all $\vartheta > 2$.
Thus the convergence rate $2\tau$ holds for $\vartheta > 2$,
resulting in the overall notation $\tau \min(\vartheta, 2)$.

In the escaping regime, where in particular $\vartheta<1$ and thus $\vartheta\tau<\tau$,
we remark that the article~\cite{narcowich2006sobolev} assumes that the kernel is translation-invariant. 
However, these results can be easily extended to any Sobolev kernel.
The details are provided in \Cref{sec:additional_proofs}.
\end{proof}

We recall the following inverse and saturation statement from \cite{wenzel2026sharp},
which summarized several works for the escaping regime \cite{schaback2002inverse,wenzel2025sharp,avesani2025sobolev} as well as the superconvergence regime 
\cite{wenzel2026sharp}.

\begin{theorem}
\label{th:result_inverse}
Under \Cref{ass:kernel_domain},
consider $f \in \mathcal{C}(\Omega)$ and the estimate 
\begin{align}
\label{eq:inverse_statement_assumption}
\Vert f - s_{f, k, X} \Vert_{L_2(\Omega)} \leq c_f h_{X}^{\vartheta \tau}
\end{align}
\begin{enumerate}[label=(\roman*)]
\item If Eq.~\eqref{eq:inverse_statement_assumption} holds for some $\vartheta \in (0, 2 - \frac{d}{2\tau}]$ for a sequence $(X_n)_{n \in \N} \subset \Omega$ of 
point sets satisfying \Cref{ass:points},
then $f \in (\ns)_{\vartheta'}$ for all $\vartheta' \in [0, \vartheta)$.
\item If Eq.~\eqref{eq:inverse_statement_assumption} holds for some $\vartheta \in [2 - \frac{d}{2\tau}, 2]$ 
for any quasi-uniform sequence $(X_n)_{n \in \N} \subset \Omega$,
then $f \in (\ns)_{\vartheta'}$ for all $\vartheta' \in [0, \vartheta)$.
\item If Eq.~\eqref{eq:inverse_statement_assumption} holds for some $\vartheta > 2$ 
for any quasi-uniform sequence $(X_n)_{n \in \N} \subset \Omega$,
then $f = 0$.
\end{enumerate}
\end{theorem}

\begin{proof}
The statement and proof are the main result of \cite{wenzel2026sharp}, 
for the special case of kernel interpolation.
Note the mild reparametrization on the convergence rate (we use $\vartheta \tau$ instead of $\beta$ to highlight the similarity to \Cref{th:result_direct}).
\end{proof}

We briefly elaborate on the sharpness of the results of \Cref{th:result_direct} and \Cref{th:result_inverse},
as already discussed in \cite{wenzel2026sharp}:
Given some $f \in \mathcal{H}_\vartheta(\Omega)$ for $\vartheta \in (\frac{d}{2\tau}, 2]$,
\Cref{th:result_direct} yields a convergence as $h_X^{\vartheta \tau}$.
On the contrary,
if some continuous function $f$ can be approximated with a rate of $h_X^{\vartheta \tau}$,
\Cref{th:result_inverse} gives that $f \in \mathcal{H}_{\vartheta'}(\Omega)$ for all $\vartheta' < \vartheta$.
Thus, both the direct and inverse results are sharp, 
and in particular a one-to-one correspondence between smoothness (in terms of the power space parameter $\vartheta$)
and approximation rate is established for the range $\vartheta \in (0, 2]$.

In order to formalize this,
we define the notion of a rate of approximation and the notion of power space smoothness in the setting of \Cref{ass:kernel_domain}:

\begin{definition}
\label{def:rate_of_approx}
We call $\beta_0 > 0$ the rate of approximation for approximating $f \in \mathcal{C}(\Omega)$ using the kernel $k$,
if for any $\beta < \beta_0$ it holds 
\begin{align*}
\Vert f - s_{f, X} \Vert_{L_2(\Omega)} \leq c_{f, \beta} h_X^\beta
\end{align*}
for some $c_{f, \beta} > 0$ and any set of quasi-uniform points $X \subset \Omega$,
and there exists no larger $\beta_0' > \beta_0$ such that this statement remains true.
\end{definition}

\begin{definition}
\label{def:power_space_smoothness}
We call $\vartheta_0 > 0$ the power space smoothness of $f \in \mathcal{C}(\Omega)$ using the kernel $k$,
if for any $\vartheta < \vartheta_0$ it holds 
\begin{align*}
f \in (\ns)_{\vartheta},
\end{align*}
and there exists no larger $\vartheta_0' > \vartheta_0$ such that this statement remains true.
\end{definition}

Using \Cref{def:rate_of_approx} and \Cref{def:power_space_smoothness},
the previously explained one-to-one correspondence can be stated as follows:
\begin{theorem}[One-to-one correspondence]
\label{th:one_to_one}
Under \Cref{ass:kernel_domain}, if $\vartheta \in (\frac{d}{2\tau}, 2]$,
then it holds for $f \in \mathcal{C}(\Omega)$:
\begin{align*}
&\text{$f$ has power space smoothness $\vartheta$ (with respect to the kernel $k$)} \\
\Leftrightarrow~
&\text{$f$ can be approximated with rate $\vartheta \tau$ (with respect to the kernel $k$)}
\end{align*}
\end{theorem}

For $\vartheta > 2$,
the saturation statement of \Cref{th:result_inverse} holds:
It states that there are no non-trivial functions that can be approximated with rate larger than $2 \tau$ using all sequences of quasi-uniform points.
It is important to note that this saturation statement does not preclude having an arbitrary large convergence rate for a single, particular sequence of quasi-uniform point sets $(X_n)_{n \in \N}
\subset \Omega$. 

\begin{table}[h!]
\centering
 \begin{tabular}{||c|| c  c||} 
 \hline
 & Escaping regime & Superconvergence regime \\ [0.5ex] 
 \hline\hline
Direct statements & \cite{narcowich2005sobolev,wendland2005approximate,narcowich2006sobolev,arcangeli2007extension} & 
\cite{schaback1999improved,schaback2018superconvergence,sloan2025doubling,karvonen2025general} \\ 
Inverse statements & \cite{schaback2002inverse,wenzel2025sharp,avesani2025sobolev} & \cite{wenzel2026sharp} \\ [1ex] 
 \hline
 \end{tabular}
 \caption{Overview of direct and inverse statements for both the escaping (the native space) and superconvergence regime.}
 \label{table:overview_results}
\end{table}

\begin{figure}[t]
\setlength\fwidth{.4\textwidth}
\begin{center}
\begin{tikzpicture}[>=latex, thick]

\begin{scope}[on background layer]
  \fill[blue!10, opacity=0.4]
    (0,-2.7cm) rectangle (5.5cm,0.5cm);

  \fill[green!10, opacity=0.4]
    (5.5cm,-2.7cm) rectangle (11cm,0.5cm);
\end{scope}

\node[blue!60!black] at (3.25cm,0.8cm)
  {\large Escaping regime};

\node[green!60!black] at (8.25cm,0.8cm)
  {\large Superconvergence regime};

\draw[->] (0,0) -- (11cm,0) node [right] {$\R$};

\draw (1.5,-3pt) -- (1.5,3pt) node[above=1pt]{0};
\draw (1.5,-2pt) node[below=0pt, align=center, fill=lightgray, rounded corners, inner sep=2pt]{$L_2(\Omega)$};

\draw (3.8,-3pt) -- (3.8,3pt) node[above=0pt]{\scriptsize $\vartheta$};
\draw (3.8, -2pt) node[below=0pt, align=center, fill=lightgray, rounded corners, inner sep=2pt]{\scriptsize $(\mathcal{H}_{k}(\Omega))_{\vartheta}$};

\draw (5.5,-2pt) -- (5.5,2pt) node[above=1pt]{1};
\draw (5.5, -2pt) node[below=0pt, align=center, fill=lightgray, rounded corners, inner sep=2pt]{$\mathcal{H}_{k}(\Omega)$};

\draw (7.8,-3pt) -- (7.8,3pt) node[above=0pt]{\scriptsize $\tilde{\vartheta}$};
\draw (7.8, -2pt) node[below=0pt, align=center, fill=lightgray, rounded corners, inner sep=2pt]{\scriptsize $(\mathcal{H}_{k}(\Omega))_{\tilde{\vartheta}}$};

\draw (9.5,-3pt) -- (9.5,3pt) node[above=1pt]{2};
\draw (9.5, -2pt) node[below=0pt, align=center, fill=lightgray, rounded corners, inner sep=2pt]{$T_{k}L_2(\Omega)$}; 

\draw[->] (0,-2cm) -- (11cm,-2cm) node [right] {$\R$};

\draw (1.5,-2cm+3pt) -- (1.5,-2cm-3pt) node[below=1pt]{0};
\draw (1.5,-2cm+2pt) node[above=0pt, align=center, fill=lightgray, rounded corners, inner sep=2pt]{$L_2(\Omega)$};

\draw (3.8,-2cm+3pt) -- (3.8,-2cm-3pt) node[below=0pt]{\scriptsize $\vartheta \tau$};
\draw (3.8, -2cm+2pt) node[above=0pt, align=center, fill=lightgray, rounded corners, inner sep=2pt]{\scriptsize $H^{\vartheta \tau}(\Omega)$};

\draw (5.5,-2cm+3pt) -- (5.5,-2cm-3pt) node[below=1pt]{$\tau$};
\draw (5.5, -2cm+2pt) node[above=0pt, align=center, fill=lightgray, rounded corners, inner sep=2pt]{$H^\tau(\Omega)$};

\draw (7.8,-2cm+3pt) -- (7.8,-2cm-3pt) node[below=0pt]{\scriptsize $\tilde{\vartheta}\tau$};
\draw (7.8, -2cm+2pt) node[above=0pt, align=center, fill=lightgray, rounded corners, inner sep=2pt]{\scriptsize $H^{\tilde{\vartheta} \tau}(\Omega)$};

\draw (9.5,-2cm+3pt) -- (9.5,-2cm-3pt) node[below=1pt]{$2\tau$};
\draw (9.5, -2cm+2pt) node[above=0pt, align=center, fill=lightgray, rounded corners, inner sep=2pt]{$H^{2\tau}(\Omega)$};

\draw (1.5, -1cm) node[align=center]{\rotatebox[origin=c]{270}{$=$}};
\draw (3.8, -1cm) node[align=center]{\rotatebox[origin=c]{270}{$\asymp$}};
\draw (5.5, -1cm) node[align=center]{\rotatebox[origin=c]{270}{$\asymp$}};
\draw (7.8, -1cm) node[align=center]{\rotatebox[origin=c]{270}{$\subseteq$}};
\draw (9.5, -1cm) node[align=center]{\rotatebox[origin=c]{270}{$\subseteq$}};

\end{tikzpicture}
\end{center}
\caption{Visualization of the scale of power spaces (top arrow) and Sobolev spaces (bottom arrow).
Several special cases like $L_2(\Omega)$, $\ns \asymp H^{\tau}(\Omega)$ and $T_k L_2(\Omega) \subset H^{2\tau}(\Omega)$ are depicted,
as well as the equivalence and subset relations of Eq.~\eqref{eq:power_spaces_escaping} and \eqref{eq:power_spaces_superconv}.
}
\label{fig:visualization}
\end{figure}
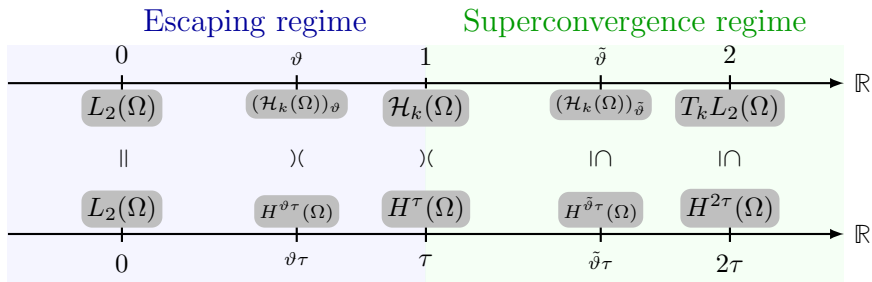

\subsection{Related work on the optimal shape parameter} %
\label{subsec:literature_optimal_shape_para}
Parameter selection is of great importance in kernel approximation, and has especially been a central topic for RBFs and their associated shape parameter.
This importance is testified by results appearing in classical monographs such as Section 
15.5 in~\cite{wendland2005scattered}, Chapters 15 and 17 in~\cite{fasshauer2007meshfree}, and Chapter 14 
in~\cite{fasshauer2015book}, the last one even including a recent overview of parameter optimization methods, also from a stochastic point of view.
Morever, Section 10.2 in~\cite{buhmann2003radial} identifies the search for the optimal shape parameter as an outstanding open problem.

A large body of literature has been devoted to develop heuristics and practical 
algorithms, with different ideas and goals.
Although early works on RBFs considered a fixed value of the shape parameter $\varepsilon$, such as with the multiquadric kernel in~\cite{hardy1971multiquadric}, 
it soon became clear that optimization of the shape parameter $\varepsilon$ was a crucial ingredient for the success RBF approximation.
The well known Rippa method~\cite{rippa1999algorithm} developed an efficient trick for the computation of Leave-One-Out Cross Validation (CV), which is widely 
used and has been extended to $k$-fold and stochastic CV~\cite{marchetti2021extension,ling2022stochastic}.
Systematic early studies investigating the effect of $\varepsilon$ and its optimization have been published 
in~\cite{fasshauer2007choosing,mongillo2011choosing}.
Further works have consider the use of global (stochastic) optimization 
methods~\cite{cavoretto2021global,mukhametzhanov2020experimental,cavoretto2024comparing}, and data driven 
estimations~\cite{ghalichi2022algorithm,veiga2026learning}.
The possible rigidity and strong dependence on the correct choice of the shape parameter has been relaxed by considering adaptive 
scaling~\cite{bozzini2002adaptive,bozzini2004scales}, more general Variably Scaled 
Kernels (VSKs)~\cite{bozzini2015interpolation,demarchi2020jumping,campi2021learning,demarchi2022persistence}, and even generalized scaling matrices
\cite{wenzel2024data,radhakrishnan2024mechanism}.
Many of these methods are concerned with the traditional accuracy-stability-complexity tradeoff, recently also using the effective condition 
number~\cite{chen2023selection,noorizadegan2024introducing,noorizadegan2022effective} as an indicator.
Beyond approximation theory, parameter optimization in kernel methods is central in the Gaussian Process community, see 
e.g.~\cite{karvonen2023maximum,karvonen2020maximum,teckentrup2020convergence}.

For a more comprehensive survey of existing computational approaches we refer also to the recent review~\cite{sun2026optimizing}, which particularly 
distinguishes three paradigms: Traditional methods (LOOCV, condition number estimation), optimization driven strategies (genetic algorithms, Bayesian 
optimization, simulated annealing), data-driven approaches (neural network based, variable shape parameter). Figure 2 in~\cite{sun2026optimizing} in particular
provides a detailed historical overview of several methods published from 1970 to the present day.

Theoretical results on RBF shape parameter optimization have been notably achieved in relation to the flat limit $\varepsilon\to 0$.
Several results have been proven in this case about the optimality of this regime for certain classes of functions and for both infinitely and 
finitely smooth kernels (see~\cite{driscoll2002interpolation,fornberg2004some,larsson2005theoretical,lee2007convergence,schaback2008limit,kindelan2016flat}
for smooth RBFs, \cite{karvonen2020worst} and~\cite{cheng2012multiquadric} specifically for the Gaussian and multiquadric kernel, 
\cite{lee2014flat,song2012multivariate} for finitely smooth kernels), also beyond interpolation (e.g. \cite{barthelme2021spectral} covers kernel matrix 
analysis and~\cite{barthelme2023gaussian} addresses Gaussian Processes in the flat limit).
Furthermore, from a functional perspective the change of shape parameter is known to change the corresponding RKHS for smooth 
kernels~\cite{zhang2013inclusion,steinwart2006explicit}, while it generally only changes the norm to an equivalent one for Sobolev kernels. This fact in 
particular is the first ingredient behind the analysis of multiscale methods~\cite{floater1996multistep,wendland2010multiscale,townsend2013multiscale}.
The kernel shape parameter(s) enter also the definition of the generalized Sobolev norms for Green 
kernels~\cite{fasshauer2011green,fasshauer2012green,fasshauer2013green}, thus offering a theoretically-motivated way to choose the value of $\varepsilon$.
Recently, the article~\cite{larsson2024scaling} introduced a comprehensive study of the effect of scaling on kernel interpolation, highlighting a number of 
phenomena that we are partially addressing also in this work.
A theoretical investigation of the interpolation error using Gaussian kernels will appear in \cite{larsson2026optimal}.

None of these works was able to theoretically explain the error behaviour based on the choice of the shape parameter. 
We are going to fill this gap by providing a solid theoretical analysis based on sharp direct and inverse statements.

\section{The optimal shape parameter and superconvergence}
\label{sec:optimal_shape_para}

In this section,
we first start by emphasizing the importance of asymptotic considerations over pre-asymptotic considerations (i.e.\ finite sets of interpolation points $X \subset \Omega$)
in order to obtain a meaningful problem.
Subsequently, 
we provide the main result \Cref{th:main_result} based on sharp direct and inverse statements,
and finally discuss the underlying assumptions and possible extensions.

Throughout the whole section, we focus on Sobolev kernels as defined in \Cref{ass:kernel_domain}.

\subsection{Non-asymptotic vs asymptotic regime}
\label{subsec:non_asymp_vs_asymp}

In this first subsection we motivate and introduce the notion of an asymptotically optimal shape parameter,
and more generally that of an asymptotically optimal kernel.

Consider input points $X = \{ x_i \}_{i=1}^n \subset \Omega$ and target values $Y = \{ y_i \}_{i=1}^n \subset \R$,
that are noiseless observations of an unknown function $f \in \mathcal{C}(\Omega)$, i.e. $f(x_i) = y_i$ for $i=1, ..., n$.
When having only access to the limited and finite data $X$ and $Y$, 
any choice $\varepsilon_0 > 0$ for the shape parameter of an RBF kernel could be optimal in terms of accuracy:
For this, imagine the unknown function $f$ is given by the kernel interpolant using the kernel $k_{\varepsilon_0}$, 
i.e.\ $f := s_{f, k_{\varepsilon_0}, X}$.
In this way, the choice $\varepsilon = \varepsilon_0$ provides an exact approximation (and is thus clearly optimal in terms of accuracy),
while other values $\varepsilon \neq \varepsilon_0$ usually do not provide an exact approximation,
as $s_{f, k_\varepsilon, X} \neq f = s_{f, k_{\varepsilon_0}, X}$.
This basic example illustrates that the search for an optimal shape parameter in terms of accuracy is ill-posed in the non-asymptotic setting,
i.e.\ if there are only finitely many input points and corresponding target values available.

Therefore we argue the need to consider the search for an optimal shape parameter in the asymptotic setting,
i.e.\ where one considers at least a sequence $(X_n)_{n \in \N} \subset \Omega$ of increasingly denser points with corresponding function values.
In this asymptotic regime, 
there are direct and inverse statements available:
Thus one obtains convergence rates for the approximation error of $f$ by $s_{f, k_\varepsilon, X}$, 
e.g.\ via bounds on $\Vert f - s_{f, k_\varepsilon, X} \Vert_{L_2(\Omega)}$, 
see \Cref{th:result_direct}.
On the contrary,
one can conclude smoothness properties of $f$ based on the convergence rate, 
see \Cref{th:result_inverse}.
This then yields the one-to-one correspondence formulated in \Cref{th:one_to_one}.

In the next subsection,
we will present how these direct and inverse statements enable a precise characterization of an asymptotically optimal shape parameter.
For this and as a preparation, keeping in mind \Cref{def:rate_of_approx}, 
we define the notion of a parameter $\varepsilon_0 \in \mathcal{P}$ being \textit{asymptotically optimal}.
Here, $\mathcal{P}$ denotes a suitable set of parameters, e.g.\ $\mathcal{P} = \R_+$ for the case of RBF kernels:

\begin{definition}
\label{def:optimal_shape}
We call $\varepsilon_0 \in \mathcal{P}$ an \textit{asymptotically optimal shape parameter},
if and only if the kernel $k_{\varepsilon_0}$ provides the fastest rate of approximation among all shape parameter $\varepsilon \in \mathcal{P}$.
\end{definition}

Furthermore, we define a kernel to be optimal if it provides the fastest possible rate of approximation.

In \Cref{subsec:preasymp_optimal_shape} we illustrate the difference between a non-asymptotic consideration and an asymptotic consideration another time:
We will highlight that there may be good shape parameters in the non-asymptotic regime, which are however not asymptotically optimal in the sense of \Cref{def:optimal_shape}.

\subsection{Sharp direct and inverse statements}
\label{subsec:sharp_direct_inverse_new}

In order to connect the sharp direct result of \Cref{th:result_direct} and the inverse as well as the saturation result of \Cref{th:result_inverse}
with the task of choosing an asymptotically optimal shape parameter, 
we note that changing the shape parameter $\varepsilon$ of a Sobolev translational invariant kernel usually does not change the corresponding RKHS as a set of functions
-- only the inner product is changed.
This follows immediately from the dilation property $\mathcal{F}[\Phi(\varepsilon ~ \cdot)] = \varepsilon^{-d} \mathcal{F}[\Phi(\cdot)](\frac{\omega}{\varepsilon})$ of the Fourier transform and the asymptotic behaviour of Eq.~\eqref{eq:fourier_decay}.
Thus the norm equivalence $\mathcal{H}_{k_\varepsilon}(\Omega) \asymp H^\tau(\Omega)$ is still valid.
We formulate this in a more general form in \Cref{ass:kernel_shape}, 
to which we will stick to in the following.
For this we make use of a general parameter set $\mathcal{P} \subset \R^p$ instead of a single real valued shape parameter:

\begin{assumption}
\label{ass:kernel_shape}
Let $\Omega \subset \R^d$ be a compact Lipschitz region.
Let $\mathcal{P} \subset \R^p$ for some $p \in \N$ be a set of parameters
such that $k_\varepsilon: \Omega \times \Omega \rightarrow \R$ is a parametrized Sobolev kernel,
i.e.\ it holds for all $\varepsilon \in \mathcal{P}$
\begin{align}
\label{eq:norm_equiv_varepsilon}
\mathcal{H}_{k_\varepsilon}(\Omega) \asymp H^\tau(\Omega).
\end{align}
\end{assumption}

\Cref{ass:kernel_shape} covers kernels frequently used in practice:
\begin{example}
\label{ex:ex1_transl_inv}
Let $k(x, z) = \Phi(x-z)$ be a translational invariant kernels satisfying Eq.~\eqref{eq:fourier_decay} and $\Omega \subset \R^d$ be a compact Lipschitz region.
Choose $\mathcal{P} = (0, \infty) \subset \R$.
Then $k_\varepsilon(x, z) := \Phi(\varepsilon(x-z))$ for $\varepsilon \in \mathcal{P}$ is a parametrized kernel that satisfies \Cref{ass:kernel_shape}.
\end{example}

\begin{example}
\label{ex:ex2_2L_kernels}
Let $k(x, z) = \Phi(x-z)$ be a translational invariant kernel satisfying Eq.~\eqref{eq:fourier_decay} and $\Omega \subset \R^d$ be a compact Lipschitz region.
Choose $\mathcal{P} = \{ A \in \R^{d \times d} ~|~ \mathrm{rank}(A) = d \} \subset \R$,
i.e.\ $d \times d$ matrices of full rank.
Then $k_A(x, z) := \Phi(A(x-z))$ for $A \in \mathcal{P}$ is a parametrized kernel that satisfies \Cref{ass:kernel_shape} \cite{wenzel2024data}.
\end{example}

\Cref{ex:ex1_transl_inv} covers in particular the case of RBF kernels,
while \Cref{ex:ex2_2L_kernels} covers the case of fully (non-degenerate) shape-adapted kernels,
see e.g.\ \cite{wenzel2024data} or \cite{radhakrishnan2024mechanism}. %
Further possible examples can be built based on VSKs,
which may even result in non-translational invariant kernels.

With these notions, we can show how the asymptotically optimal shape parameter $\varepsilon_0$ is determined by the underlying Mercer decomposition of the kernel $k_{\varepsilon_0}$.

Note that by \Cref{ass:kernel_shape}, 
for every $k_\varepsilon$ with $\varepsilon \in \mathcal{P}$ we are in the setting of \Cref{ass:kernel_domain},
so that the direct and inverse theorems of \Cref{subsec:sharp_direct_inverse} are applicable.
Note that a change of the shape parameter $\varepsilon$ changes the Mercer decomposition as introduced in \Cref{subsec:kernels_rkhs_power}.
Indeed, given a parametrized kernel $k_\varepsilon$ such that $\mathcal{H}_{k_\varepsilon}(\Omega) \asymp H^\tau(\Omega)$,
an application of Mercer theorem to the kernel integral operator $T_{k_\varepsilon}: L_2(\Omega) \rightarrow L_2(\Omega)$ defined as
\begin{align}
\label{eq:integral_op_varepsilon}
(T_{k_\varepsilon} f)(x) := \int_{\Omega} k_\varepsilon(x, z)f(z) ~ \mathrm{d}z	
\end{align}
yields eigenvalues $(\lambda_{j,\varepsilon})_{j \in \N}$ and eigenfunctions $(\varphi_{j, \varepsilon})_{j \in \N} \subset L_2(\Omega)$,
that may depend on $\varepsilon \in \mathcal{P}$.
This thus also changes the Mercer-based power spaces,
i.e.\ for the kernel $k_\varepsilon$ the power spaces are given by
\begin{align}
\label{eq:power_space_varepsilon}
(\mathcal{H}_{k_\varepsilon}(\Omega))_\vartheta := \left\{ f \in L_2(\Omega) ~:~ \sum_{j=1}^\infty \frac{|\langle f, \varphi_{j, \varepsilon} \rangle_{L_2(\Omega)}|^2}{\lambda_{j, \varepsilon}^\vartheta} < \infty \right\} \subset L_2(\Omega).
\end{align}
Note that Eq.~\eqref{eq:norm_equiv_varepsilon} implies that $\lambda_{j, \varepsilon} \asymp j^{-2\tau/d}$,
thus the $\varepsilon$-dependency of the summability within Eq.~\eqref{eq:power_space_varepsilon} essentially only depends on the asymptotic behaviour of the inner products $|\langle f, \varphi_{j, \varepsilon} \rangle_{L_2(\Omega)}|$.

Essentially,
this chain of thoughts can be summarized in the following theorem:

\begin{theorem}[Main result]
\label{th:main_result}
Under \Cref{ass:kernel_shape}, consider $0 \neq f \in \mathcal{C}(\Omega)$.
Let $\varepsilon \in \mathcal{P}$ and let $\beta _0 > 0$ be the rate of approximation  of $f$ by interpolation with $k_\varepsilon$, i.e., for any $\beta < \beta_0$ and all quasi-uniformly distributed points $X \subset \Omega$ it holds that
\begin{align}
\label{eq:main_result_conv_rate}
\Vert f - s_{f, k_\varepsilon, X} \Vert_{L_2(\Omega)} \leq c_{f, \beta} h_X^{\beta}.
\end{align}
Then $\beta_0 = \min(\vartheta_0 \tau, 2 \tau)$,
with $\vartheta_0$ being the largest value such that it holds
\begin{align}
\label{eq:main_result_largest_vartheta}
\sum_{j=1}^\infty j^{2\vartheta \tau/d} |\langle f, \varphi_{j, \varepsilon} \rangle_{L_2(\Omega)}|^2 < \infty
\end{align}
for all $\vartheta < \vartheta_0$.
Here $\varphi_{j, \varepsilon}$ are the Mercer eigenfunctions of the operator defined in Eq.~\eqref{eq:integral_op_varepsilon}
\end{theorem}

We should assume $\vartheta > \frac{d}{2\tau}$,
or extend the direct statements to $\vartheta \in (0, \frac{d}{2\tau} \}$. 
I'll think about it.
$\rightarrow$ I checked \cite{narcowich2006sobolev},
but I did not see a quick way to extend the direct escaping statements to $(0, \frac{d}{2\tau})$.
Thus I specialized \Cref{subsec:no_optimal_in_escaping} to $\vartheta > \frac{d}{2\tau}$.

\begin{proof}
Due to \Cref{ass:kernel_shape}, any kernel $k_\varepsilon$ is a Sobolev kernel.
Thus we obtain for any $\varepsilon \in \mathcal{P}$ a Mercer expansion with ordered eigenvalues
$(\lambda_{j, \varepsilon})_{j \in \N}$ with asymptotics $\lambda_{j, \varepsilon} \asymp j^{-2\tau/d}$ and eigenfunctions $(\varphi	_{j, \varepsilon})_{j \in \N}$.
Since $f \in \mathcal{C}(\Omega) \hookrightarrow L_2(\Omega)$,
there exists a value $\vartheta_0 \geq 0$ such that Eq.~\eqref{eq:main_result_largest_vartheta} holds for all $\vartheta < \vartheta_0$.
By the direct result \Cref{th:result_direct} and \Cref{def:rate_of_approx},
this implies a rate of approximation of $\tau \min(\vartheta_0, 2)$.
Furthermore, the rate of approximation cannot be larger because of the inverse result \Cref{th:result_inverse}:
Any larger rate of approximation such that Eq.~\eqref{eq:main_result_conv_rate} holds,
implies that $f$ needs to be in a smoother power space,
contradicting the fact that $\vartheta_0$ was chosen as the largest value such that Eq.~\eqref{eq:main_result_largest_vartheta} holds.
\end{proof}

Note that the optimal shape parameter in the sense of \Cref{def:optimal_shape} is the one which provides the fastest asymptotic decay of the Mercer expansion coefficients $|\langle f, \varphi_{j, \varepsilon} \rangle_{L_2(\Omega)} |$.
This could be rephrased as choosing the shape parameter such that the spectral decay is fastest.

On the other hand,
as the power spaces $(\mathcal{H}_{k_\varepsilon})_\vartheta$ for $\vartheta > 1$ are conjectured to be characterized by the Sobolev smoothness $\vartheta \tau$ and additional boundary conditions \cite{karvonen2025general},
this also means that the shape parameter $\varepsilon$ should be chosen such that the resulting boundary conditions of $(\mathcal{H}_{k_\varepsilon})_\vartheta$ match the boundary conditions of the function $f$ of interest.

\subsection{Discussion of the main results}
\label{subsec:disussion}

The main result \Cref{th:main_result} is formulated for kernel interpolation of some function $f \in \mathcal{C}(\Omega)$ 
using a parametrized family of Sobolev kernels and quasi-uniformly distributed points.
We now discuss these assumptions and how the result can be extended beyond them.
Since \Cref{th:main_result} mainly relies on the combination of sharp direct and inverse statements,
we mainly need to comment on those:

Beyond kernel \emph{interpolation}:
Most of the direct, inverse and saturation statements actually hold beyond kernel interpolation,
thus the theoretical framework applies to these more general cases:
In fact, the inverse statements in \cite{wenzel2026sharp} for both the escaping as well as the superconvergence regimes are already formulated for general point-based approximants $s_{f, k, X} \in \Sp \{ k(\cdot, x_i), x_i \in X\}$,
thus covering also e.g.\ regularized approximation or Galerkin approximation.
Direct statements for the escaping regime are also available beyond interpolation \cite{wendland2005approximate},
only the direct statements in the superconvergence regime currently rely explicitly on orthogonal approximation (e.g.\ by interpolation) \cite{karvonen2025general}.
Like this, our theory is expected to be applicable also beyond kernel interpolation,
with possible applications to e.g.\ regularized approximation, statistical learning theory or kernel based methods for inverse problems.
We decided to state the main result focusing on kernel interpolation in order to keep the focus on the overall theoretical advancement instead of delving into too many technical details.

Beyond \emph{$L_2(\Omega)$ errors}:
Most of the direct statements are already available in general $\Vert \cdot \Vert_{L_p(\Omega)}$ norms for $1 \leq p \leq \infty$,
though it remains unclear whether these are sharp for $1 \leq p \leq 2$.
Unfortunately, there are only few inverse statements available using $p \neq 2$ \cite{schaback2002inverse,wenzel2026sharp},
and the available ones are not sharp.
One partial reason for this is the fact that the $\Vert \cdot \Vert_{L_p(\Omega)}$ norms for $p \neq 2$ are no longer Hilbert spaces norms,
thus making the inverse analysis more challenging.
Instead of Hilbert spaces,
one likely needs to consider Besov spaces \cite{wenzel2026sharp}.

Beyond \emph{quasi-uniform} point distributions: 
The assumption on quasi-uniform point distributions is mainly required for stability reasons:
Direct statements in the escaping regime require such stability properties \cite{narcowich2006sobolev},
and the inverse statements in both regimes also rely on such stability properties \cite{wenzel2025sharp,avesani2025sobolev,wenzel2026sharp}.
However, such stability properties can be obtained by other means than quasi-uniform point distributions,
e.g.\ by regularization \cite{wendland2005approximate}. %
Furthermore, also randomly distributed points are known to be (close to) optimal for approximation \cite{krieg2024random},
while not necessarily hurting stability arguments too severely \cite{buchholz2022kernel}.
We leave the examination of the resulting technical details and implications for future work.

Beyond \emph{finitely} smooth kernels: 
For infinitely smooth kernels like the Gaussian kernel, 
the situation is crucially different.
Here, the available stability estimates do not match the corresponding approximation error rates \cite{rieger2010sampling,diederichs2019improved},
such that the analyis for the inverse statements cannot easily be transfered.
Furthermore,
quasi-uniform points do not seem to be optimal \cite{rieger2014improved,wenzel2021novel}.
Nevertheless, superconvergence phenomenon are still available 
-- however in view of exponential convergence rates with undetermined constants,
it remains unclear whether a potential \emph{doubling} of the rate is really beneficial at all.
Finally, 
note that e.g.\ for the Gaussian kernel, 
the shape parameter $\varepsilon > 0$ actually steers the smoothness:
Flatter kernels give smaller RKHSs, i.e.\ smoother functions \cite[Corollary 7]{steinwart2006explicit}.
We consider the analyis for infinitely smooth kernels to be another challenging task,
probably requiring different theoretical tools.
Some steps in this direction have been done in \cite[Section 6]{larsson2024scaling}.

In summary,
\Cref{th:main_result} can likely be extended to a broader setting,
and this is confirmed in numerical experiments.
This will be adressed in follow-up works.

\section{Implications and Applications}
\label{sec:implications_applications}

In this section, we collect several immediate implications and applications of our main result \Cref{th:main_result}.
We are always working in the setting of \Cref{ass:kernel_shape}.

\subsection{No optimal kernel in the escaping regime}
\label{subsec:no_optimal_in_escaping}

As a first consequence of the main result \Cref{th:main_result}, 
we conclude that there is no asymptotically optimal shape parameter in the escaping regime.
Given $f \in H^{\vartheta \tau}(\Omega)$ for $\vartheta \in (0, 1)$, any kernel $k_\varepsilon$ satisfies $\mathcal{H}_{k_\varepsilon}(\Omega) \asymp H^\tau(\Omega)$ by \Cref{ass:kernel_shape},
and furthermore the power spaces $(\mathcal{H}_{k_\varepsilon}(\Omega))_\vartheta$ for $\vartheta \in [0, 1]$ are given by the Sobolev spaces $H^{\vartheta \tau}(\Omega)$ of intermediate smoothness.
Then the convergence rate is exactly characterized by the Sobolev smoothness, i.e.\ $\vartheta \tau$,
and thus independent of the choice of the shape parameter $\varepsilon \in \mathcal{P}$.
Thus, all shape parameters provide the same rate of approximation, 
and there is no shape parameter that provides a faster (or slower) convergence rate.

\begin{theorem}
Under \Cref{ass:kernel_shape}, consider $f \in \mathcal{C}(\Omega)$ of Sobolev smoothness $\vartheta \tau \in (\frac{d}{2}, \tau]$ (i.e.\ escaping regime),
i.e.\ it holds $f \in H^\sigma(\Omega)$ for any $\sigma < \vartheta \tau$ 
but $f \notin H^{\sigma'}(\Omega)$ for any $\sigma' > \vartheta \tau$.

Then the kernel interpolant $s_{f, k_\varepsilon, X}$ for any $\varepsilon \in \mathcal{P}$ yields the rate of approximation $\vartheta \tau$,
i.e.\ there is no asymptotically optimal shape parameter $\varepsilon \in \mathcal{P}$.
\end{theorem}

\begin{proof}
By \Cref{ass:kernel_shape} and Eq.~\eqref{eq:power_spaces_escaping},
for any $\varepsilon \in \mathcal{P}$
it holds $(\mathcal{H}_{k_\varepsilon}(\Omega))_\vartheta \asymp H^{\vartheta\tau}(\Omega)$ for all $\vartheta \in [0, 1]$.
Due to the assumption of $f$ being in exactly this escaping regime of $\vartheta \in [0, 1]$,
the direct statement \Cref{th:result_direct} and inverse statement \Cref{th:result_inverse} imply a rate of approximation of $\vartheta \tau$.
\end{proof}

In fact,
the same argument works for any Sobolev kernel as defined in \Cref{ass:kernel_domain},
such that one can conclude not only that there is no optimal shape parameter,
but even that there is no asymptotically optimal kernel
in the escaping regime.

\subsection{Optimal kernels in the superconvergence regime} 
\label{subsec:existence_of_superconv_kernel}

As a second consequence,
we show that there always exists an optimal kernel in the superconvergence regime.

\begin{theorem}
\label{th:existence_optimal_kernel}
Under \Cref{ass:kernel_shape},
consider $f \in \mathcal{C}(\Omega)$ of power space smoothness $\vartheta \tau \in (\tau, 2\tau]$ (i.e.\ superconvergence regime),
i.e.\ it holds $f \in \calh_{\theta}(\Omega)$ for $\theta < \vartheta \tau$,
but $f \notin \calh_{\theta'}(\Omega)$ for any $\theta' > \vartheta \tau$.

Then there are kernels $k$ such that $\ns \asymp H^\tau(\Omega)$ and that provide a rate of approximation $\vartheta \tau$.
\end{theorem}
\begin{proof}
Since $\vartheta\tau > \tau > d/2$ the space $H^{\vartheta \tau}(\Omega)$ is a RKHS. 
Let $k_0$ be any of its kernels, i.e.  $\calh_{k_0}(\Omega) \asymp 
H^{\vartheta\tau}(\Omega)$, and let $k_0(x, z) \coloneqq \textstyle \sum_{j=1}^\infty \lambda_j\varphi_j(x)\varphi_j(z)$ be its Mercer decomposition. 
In particular
\begin{equation}\label{eq:tmp_interp}
H^{\vartheta\tau}(\Omega)\asymp\calh_{k_0}(\Omega)
=\left\{ g \in L_2(\Omega) ~:~ \sum_{j=1}^\infty \frac{| \langle g, \varphi_j \rangle_{L_2(\Omega)} |^2}{\lambda_j} < \infty \right\}.
\end{equation}
We know that the corresponding power space $\left(\calh_{k_0}(\Omega)\right)_\sigma$ with $\sigma\in(0,1)$ is norm equivalent to the
Sobolev space $H^{\sigma\vartheta \tau}(\Omega)$, and for $\sigma\vartheta\tau>d/2$ it is a RKHS with kernel $k_{0,\sigma}(x, z) \coloneqq \sum_{j=1}^\infty
\lambda_j^\sigma\varphi_j(x)\varphi_j(z)$. 
In particular, for $\sigma\coloneqq 1/\vartheta$ we get $\left(\calh_{k_0}(\Omega)\right)_{1/\vartheta} \asymp 
H^\tau(\Omega)$ and the kernel 
\begin{align*}
k(x, z)\coloneqq k_{0,1/\vartheta}(x, z) \coloneqq \sum_{j=1}^\infty
\lambda_j^{\frac1{\vartheta}}\varphi_j(x)\varphi_j(z).
\end{align*}
Denoting as $\ns=\left(\calh_{k_0}(\Omega)\right)_{1/\vartheta}$ its RKHS, for $\sigma\in(1,2)$ we 
have the power space
\begin{equation*}
(\ns)_{\sigma}=\left\{ g \in L_2(\Omega) ~:~ \sum_{j=1}^\infty \frac{| \langle g, \varphi_j \rangle_{L_2(\Omega)}|^2}{\lambda_j^{\sigma/\vartheta}}< \infty\right\},
\end{equation*}
implying that $(\ns)_{\vartheta}=\calh_{k_0}(\Omega)\asymp H^{\vartheta\tau}(\Omega)$ by Eq.~\eqref{eq:tmp_interp}. 
In summary, $k$ is a Sobolev kernel of smoothness $\tau$ and $f\in H^{\vartheta \tau}(\Omega)=(\ns)_{\vartheta}$, so interpolation of $f$ by $k$ on quasi-uniform points gives an 
$L_2(\Omega)$ rate of approximation $\vartheta\tau$.
\end{proof}

We note that $k(x, z) := f(x)f(z)$ is an even better kernel for approximating $f$,
however this is not a Sobolev kernel as it is not strictly positive definite.
In \cite[Remark 28]{karvonen2025general} it was pointed out
that one has the optimal convergence rate vs stability ratio when working in superconvergence regime.
In this respect, choosing the ``correct'' shape parameter (or kernel) can be seen as an advantage in terms of stability:
There is no need to pick a smoother kernel -- just pick the correct shape such that the function of interest is in the superconvergence regime of that kernel, if available.

Despite there always exists an optimal Sobolev kernel in the sense of \Cref{th:existence_optimal_kernel},
this kernel may not be obtainable by tuning some shape parameters of a given kernel.
This will be highlighted in the subsequent subsection.

\subsection{No, unique, or multiple optimal shape parameter}
\label{subsec:no_unique_multiple}

In this subsection we highlight 
that there may not exist an optimal shape parameter,
or there might be a unique optimal shape parameter,
or there might even be multiple optimal ones.

These cases are presented through three one-dimensional kernels, which are the Green kernel of the PDE 
\begin{equation}\label{eq:1d_pde}
-u''(x) + \varepsilon^2 u(x) = f(x), \;\; x\in\Omega\coloneqq(0,1),
\end{equation}
each one corresponding to a different set of boundary conditions (BCs). Namely, we consider the following kernels:
\begin{itemize}
\item The $\cosh$ kernel 
\begin{equation*}
k^1(x, z) \coloneqq \frac{\cosh(\varepsilon \min(x, z)) \cosh(\varepsilon(1-\max(x,z))}{\varepsilon \sinh(\varepsilon)},
\end{equation*}
corresponding to the $\varepsilon$-independent Neumann BCs 
\begin{equation}\label{eq:cosh_bcs}
u'(0) = 0, u'(1) = 0.
\end{equation}
\item The periodic kernel
\begin{align*}
k^2(x,z) &\coloneqq
\frac{
\Bigl(\cosh(\varepsilon m) + \frac{\alpha(\varepsilon)}{\varepsilon} \sinh(\varepsilon m)\Bigr)
\Bigl(\cosh(\varepsilon (1-M)) + \frac{\alpha(\varepsilon)}{\varepsilon} \sinh(\varepsilon (1-M))\Bigr)
}{
2\,\alpha(\varepsilon)\,\cosh(\varepsilon) + \Bigl(\varepsilon + \frac{\alpha(\varepsilon)^2}{\varepsilon}\Bigr)\sinh(\varepsilon)
}\\
\alpha(\varepsilon) &= 1+\cos(\varepsilon), \;\;m = \min(x,z), \;\; M = \max(x,z),
\end{align*}
corresponding to
\begin{align}\label{eq:periodic_bcs}
u'(0) - (1 + \cos(\varepsilon)) u(0) &=0\\
u'(1) + (1 + \cos(\varepsilon)) u(1) &=0.\nonumber
\end{align}
Observe that $k^2$ is not periodic in $\varepsilon$, but the boundary conditions are.
\item The basic Mat\'ern kernel $k^3(x,z)\coloneqq \frac{1}{2\varepsilon}e^{-\varepsilon|x-z|}$, with Robin BCs
\begin{align}\label{eq:matern_bc}
u'(0) - \varepsilon u(0) &=0\\
u'(1) + \varepsilon u(1) &=0.\nonumber
\end{align}
\end{itemize}
The kernels $k^1$ and $k^3$ are known to be Green kernels of Eq.~\eqref{eq:1d_pde} with the corresponding BCs (see e.g.~\cite{fasshauer2015book}), while we
obtained $k^2$ by explicit computation by prescribing Eq.~\eqref{eq:periodic_bcs}.
Given that these kernels are Green kernels of Eq.~\eqref{eq:1d_pde}, integration by parts shows that they all lead to a RKHS which is norm-equivalent to 
$H^1(0,1)$. 
The three kernels are visualized in Figure~\ref{fig:ex1_viz_kernels}:
Although they all have two similar exponential branches (due to the 
fundamental solutions of the PDE), for small $\varepsilon$ (left plot) they differ on their behavior close to the boundary because of the different boundary
conditions. When $\varepsilon$ becomes instead sufficiently large (right plot) they localize and become essentially identical, as they all converge to the Mat\'ern
kernel, which is the full-space Green kernel of the PDE.

\begin{figure}
\centering
\begin{tabular}{cc}
\includegraphics[width=.45\textwidth]{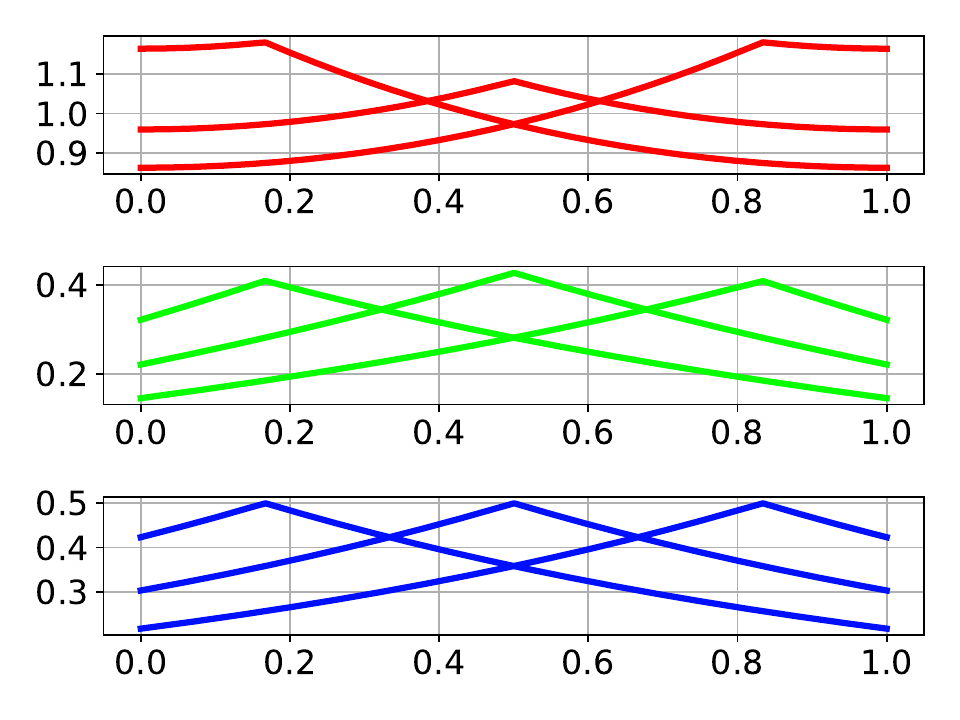}&
\includegraphics[width=.45\textwidth]{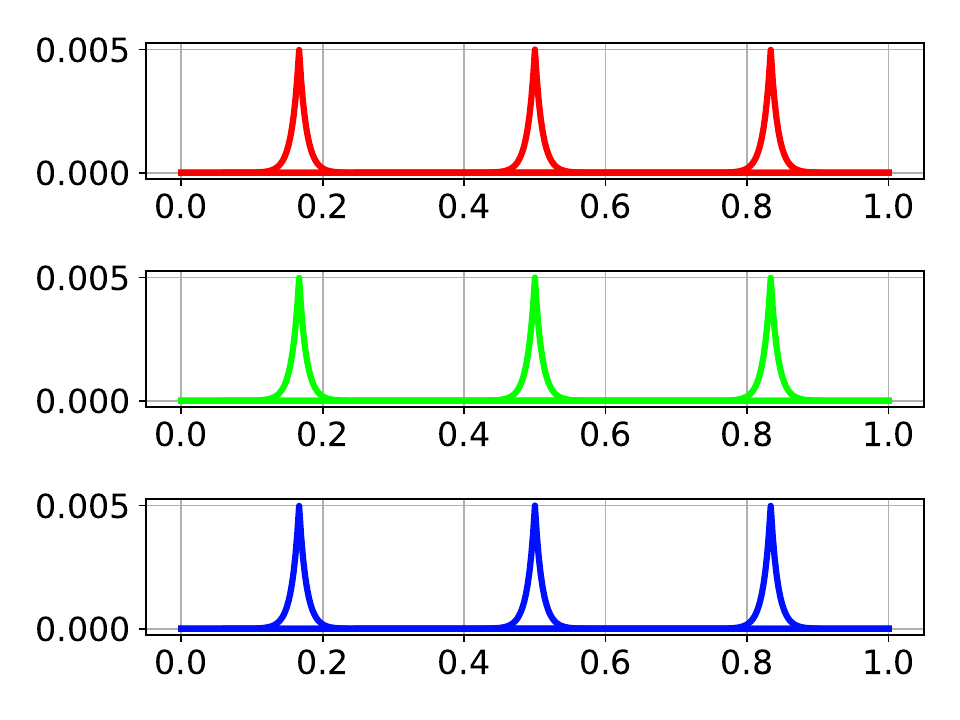}\\
\end{tabular}
\caption{Visualization of the three kernels $k^1(\cdot, x),k^2(\cdot, x), k^3(\cdot, x)$ (from top to bottom), centered at $x=1/6,1/2,5/6$, and
scaled with $\varepsilon=1$ (left) and $\varepsilon=100$ (right). 
}\label{fig:ex1_viz_kernels}
\end{figure}

Using these kernels we can prove the following result.
\begin{lemma}\label{lemma:1d}
Let $\Omega\coloneqq (0,1)$ and $k^1,k^2,k^3$ as defined above.

Let $f\in H^1(\Omega)$ and assume there exists $\bar\varepsilon>0$ and $c_f>0$ such that for any set of quasi-uniform points $X\subset\Omega$ it holds
\begin{equation}\label{eq:1d_fast_rate}
\norm{L_2(\Omega)}{f-s_{f, k_{\bar\varepsilon},X}}\leq c_f h_{X}^2,
\end{equation}
where $s_{f, k_{\bar \varepsilon}, X_n}$ is the interpolant of $f$ on $X_n$ computed with the kernel $k\in\{k^1,k^2,k^3\}$ with parameter $\bar\varepsilon$. Then we have
the following:
\begin{itemize}
 \item If $k=k^1$, then Eq.~\eqref{eq:1d_fast_rate} holds for all $\varepsilon>0$.
 \item If $k=k^2$, then Eq.~\eqref{eq:1d_fast_rate} holds for all $\varepsilon=\bar\varepsilon+2 m \pi$, $m\in\Z$.
 \item If $k=k^3$, then Eq.~\eqref{eq:1d_fast_rate} holds only for $\varepsilon=\bar\varepsilon$.
\end{itemize}
\end{lemma}
\begin{proof}
As all kernels give $\ns\asymp H^1(\Omega)$, applying Theorem~\ref{th:main_result} gives $f\in  
\calh_{k_{\bar\varepsilon},2}(\Omega)=T_{k_{\bar\varepsilon}}(L_2(\Omega))$ (see Eq.~\eqref{eq:power_space_varepsilon} and Eq.~\eqref{eq:integral_op_varepsilon} for 
the definition of these spaces).

Now each kernel is a Green kernel and thus, as observed in Section 4.3 in ~\cite{karvonen2025general}, the operator $T_{k_\varepsilon}$ 
 is the solution operator of the PDE Eq.~\eqref{eq:1d_pde}, each time with a different set of BCs. It follows that $T_{k_{\bar\varepsilon}}(L_2(\Omega))$ is the set of
solutions of this PDE. Then, under the assumptions of the lemma, $f$ must satisfy these boundary conditions for $\varepsilon=\bar\varepsilon$.

Looking at the precise definition of these BCs, this implies that $f$ must satisfy these BCs for all $\varepsilon>0$ for $k^1$ (see Eq.~\eqref{eq:cosh_bcs}),
for all $\varepsilon=\bar\varepsilon+2m\pi$, $m\in\Z$ for $k^2$ (see Eq.~\eqref{eq:periodic_bcs}), and only for $\varepsilon=\bar\varepsilon$ for $k^3$
(by Eq.~\eqref{eq:matern_bc}).
Finally, since for these values we have $f\in T_{k_\varepsilon}(L_2(\Omega))$ then Theorem 5.1 in~\cite{schaback1999improved} implies 
that Eq.~\eqref{eq:1d_fast_rate} holds.
\end{proof}

We discuss now two examples presenting different instances of the behaviors described in~\Cref{lemma:1d}. Both of them report interpolation errors obtained for 
a specific function and with the three kernels, with
$101$ logarithmically equally spaced values $\varepsilon$ in $[10^{-3}, 10^2]$, and on sets of increasingly large equally spaced points $X\coloneqq\{(2i-1)/2n: 1\leq i\leq n\}$ with $n \in \{100, 196, 387, 762, 1500\}$.
We choose to exclude the boundary points $\{0,1\}$ to avoid interpolation of the Dirichlet part of the boundary conditions.

\begin{example}\label{ex:no_dip}
Let $f(x)\coloneqq 1 + x ^\frac32 - \frac{3}{4} x ^ 2 \in H^1(\Omega)$, which satisfies the BCs Eq.~\eqref{eq:cosh_bcs} of $k^1$ for all $\varepsilon>0$, the
BCs Eq.~\eqref{eq:periodic_bcs} of $k^2$ for $\varepsilon = \pm \pi + 2 m \pi$, $m\in\N$, and does not satisfy the BCs Eq.~\eqref{eq:matern_bc} of $k^3$ for any
$\varepsilon>0$, except for the limit $\varepsilon\to 0$.

Interpolation results as functions of $\varepsilon$ are shown in Figure~\ref{fig:ex1_1d}: 
The left column reports the $L_2(\Omega)$ error (obtained via the Root Mean Squared Error (RMSE) computed on
$10^4$ 
equally spaced points), 
for the given values of $|X|$. For the same values of $|X|$, the right column shows the estimated rates of convergence. 
\begin{figure}[h!]
\centering
\begin{tabular}{cc}
\includegraphics[width=.5\textwidth]{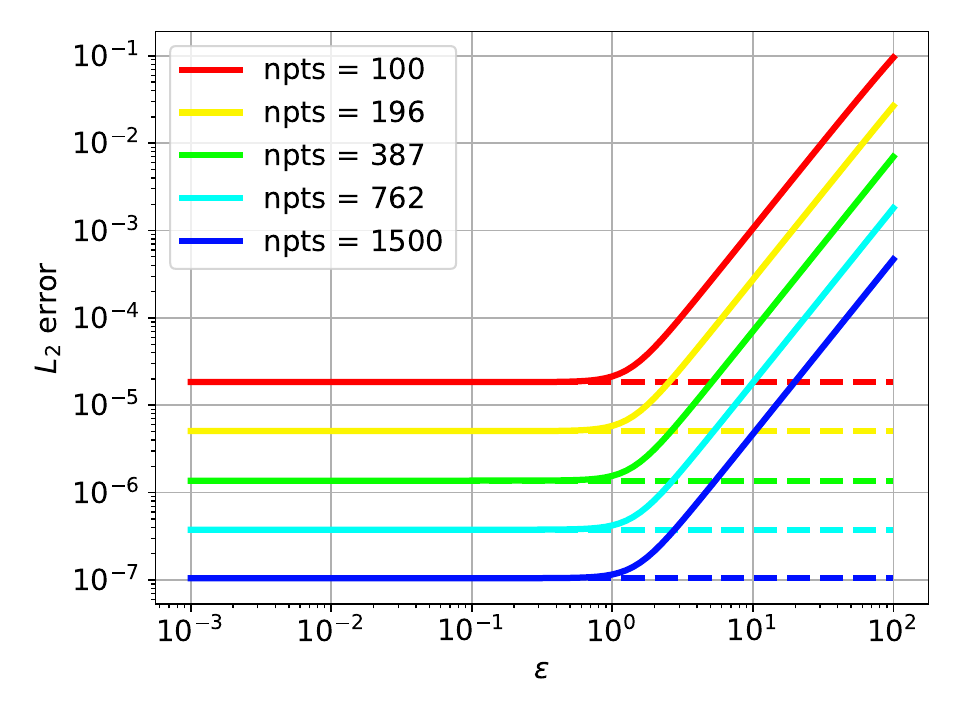}&
\includegraphics[width=.5\textwidth]{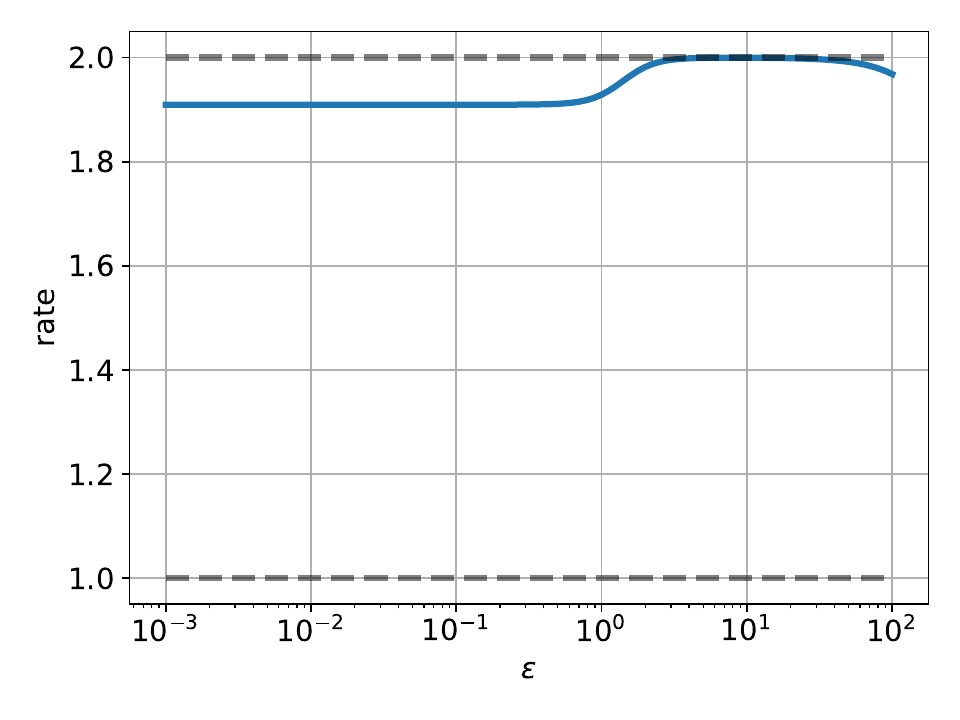}\\
\includegraphics[width=.5\textwidth]{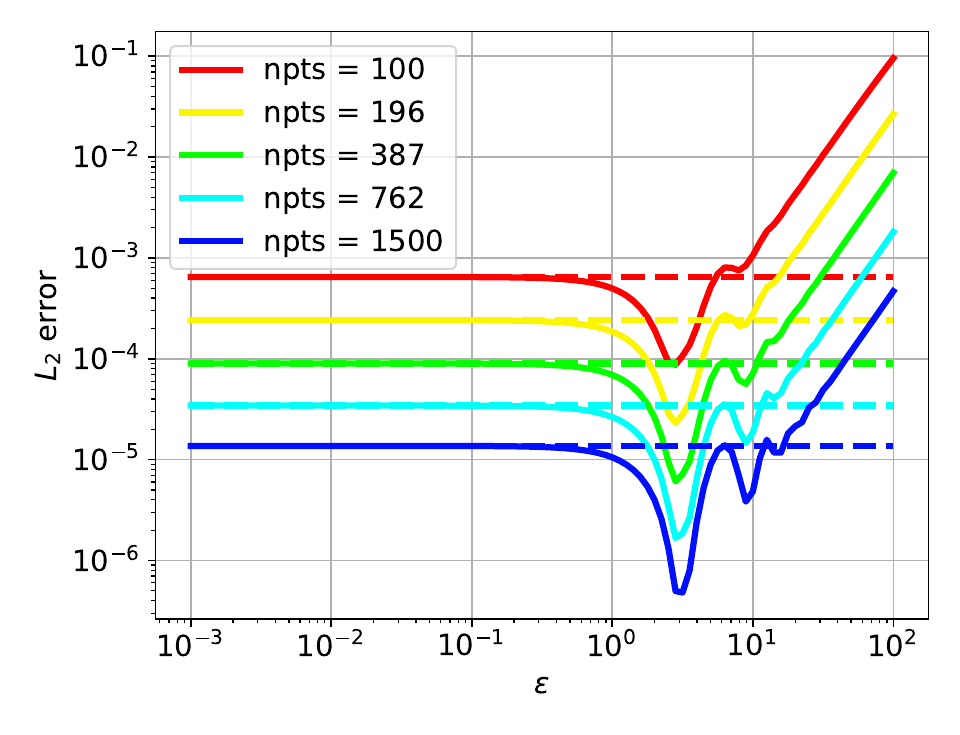}&
\includegraphics[width=.5\textwidth]{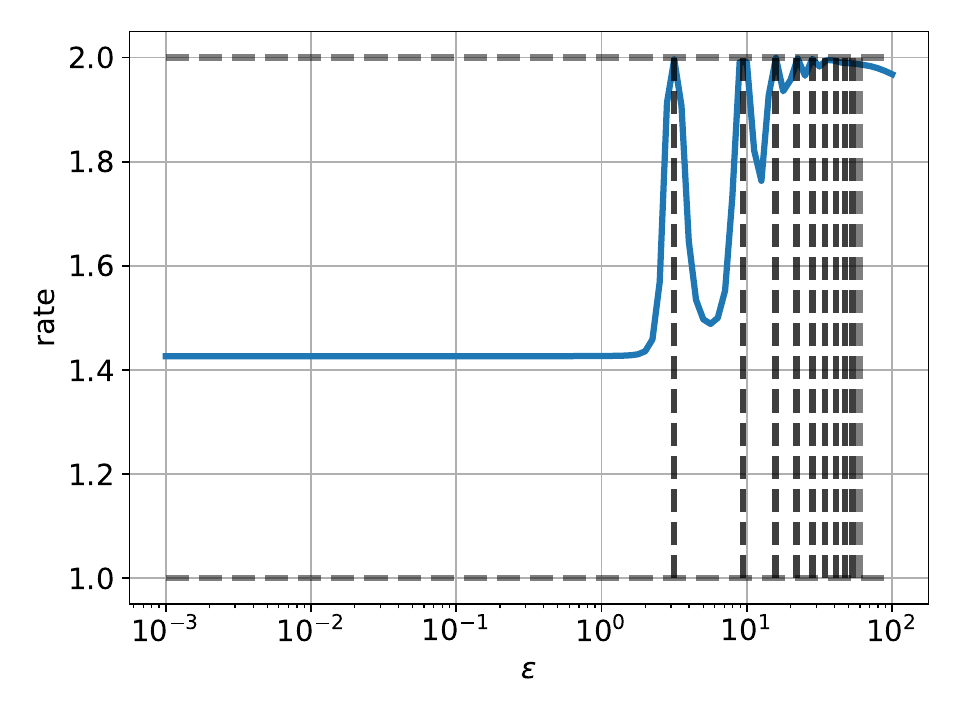}\\
\includegraphics[width=.5\textwidth]{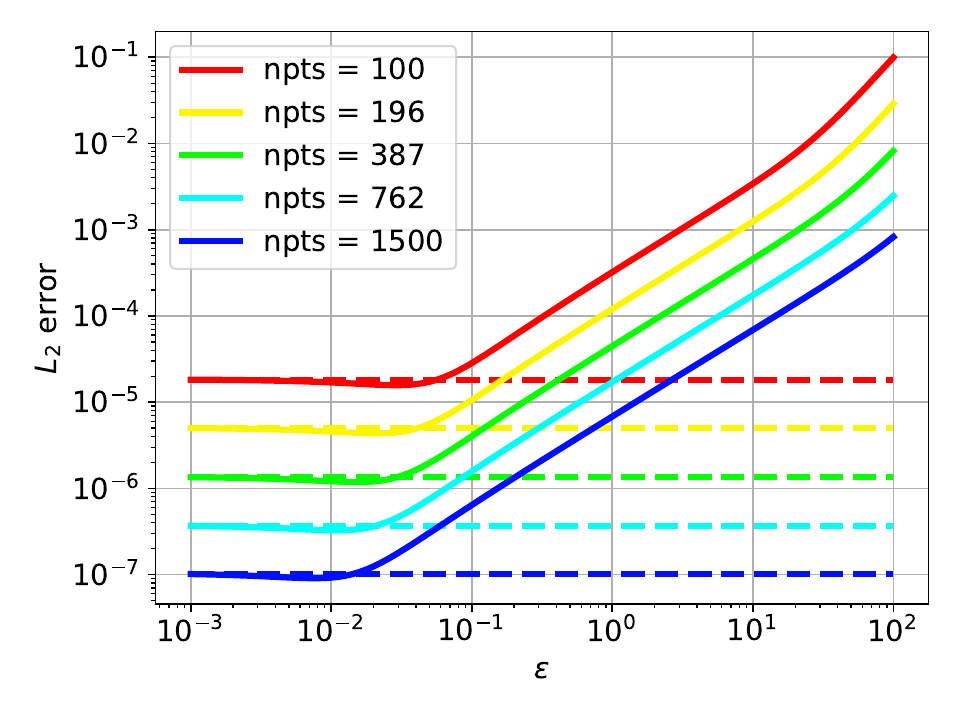}&
\includegraphics[width=.5\textwidth]{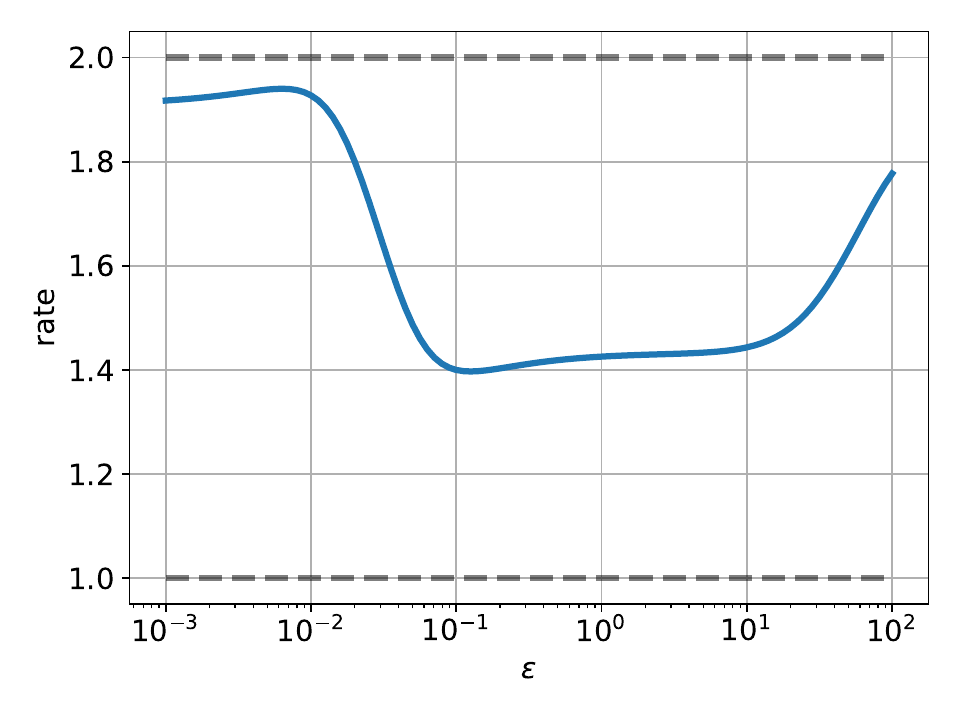}\\
\end{tabular}
\caption{Example~\ref{ex:no_dip}: Error decay (first column) and rate of decay (second column) as a function of the shape parameter $\varepsilon$ for increasing number of points $|X|$, and for the 
three kernels. 
The first column additionally reports the interpolation error in the limit $\varepsilon\to0$ extended to all 
$\varepsilon$ (dotted lines).
The second column shows horizontal lines for the basic rate in $\ns$ and the full superconvergence rate, and vertical lines for the expected optimal shape 
parameter(s), when existing.}\label{fig:ex1_1d}
\end{figure}
First, observe that our theory is mainly addressing the asymptotic behavior, thus the right column in the plots. Here we can clearly see that any value of 
$\varepsilon$ gives (almost) the optimal rate $2$ for $k^1$, which is achieved only for the predicted values of $\varepsilon$ for $k^2$ (vertical dotted bars),
and only in the limit $\varepsilon\to0$ for $k^3$. Observe moreover that for large enough $\varepsilon$ the rate is increasing, also for $k^3$, possibly
because of the localization happening in this regime (see Figure~\ref{fig:ex1_viz_kernels}). For intermediate values of $\varepsilon$ and for $k^2$ and $k^3$,
the rate stabilizes at around $1.5$, which is $1/2$ more than the basic rate $1$, and which is the conjectured rate for smooth functions which do not satisfy 
any boundary condition (see~\cite{karvonen2025general}).

Looking at the errors for fixed $|X|$ (left column), we immediately see two limiting behaviors: small values $\varepsilon$ lead to a plateau, which is 
highlighted by dashed horizontal lines and which we expect to be connected to the flat limit~\cite{song2012multivariate}. For large $\varepsilon$ instead we 
see that all kernels give larger errors, again possibly because of the high localization. For intermediate values, only $k^2$ admits optimal parameters, as
expected.
In particular, no rate better than the flat limit is observed for $k^1, k^3$, suggesting that one should use the limiting polyharmonic interpolant when
available.

\end{example}

\begin{example}\label{ex:dip}
In this case we consider $f(x)\coloneqq  x ^\frac32 + x ^ 2(-1 + \frac1{2 \bar \varepsilon + 4}) \in H^1(\Omega)$ with $\bar \varepsilon\coloneqq 1$. It does 
not satisfy the BCs of $k^1$, it satisfies those of
$k^3$ for $\varepsilon=\bar\varepsilon$, and those of $k^2$ for $\varepsilon = \pm \frac{\pi}2 + 2 m \pi$, $m\in\N$. In the same setting as
Example~\ref{ex:no_dip}, the $L_2(\Omega)$ errors and corresponding rates are shown in Figure~\ref{fig:ex2_1d}.
\begin{figure}[h!]
\centering
\begin{tabular}{cc}
\includegraphics[width=.5\textwidth]{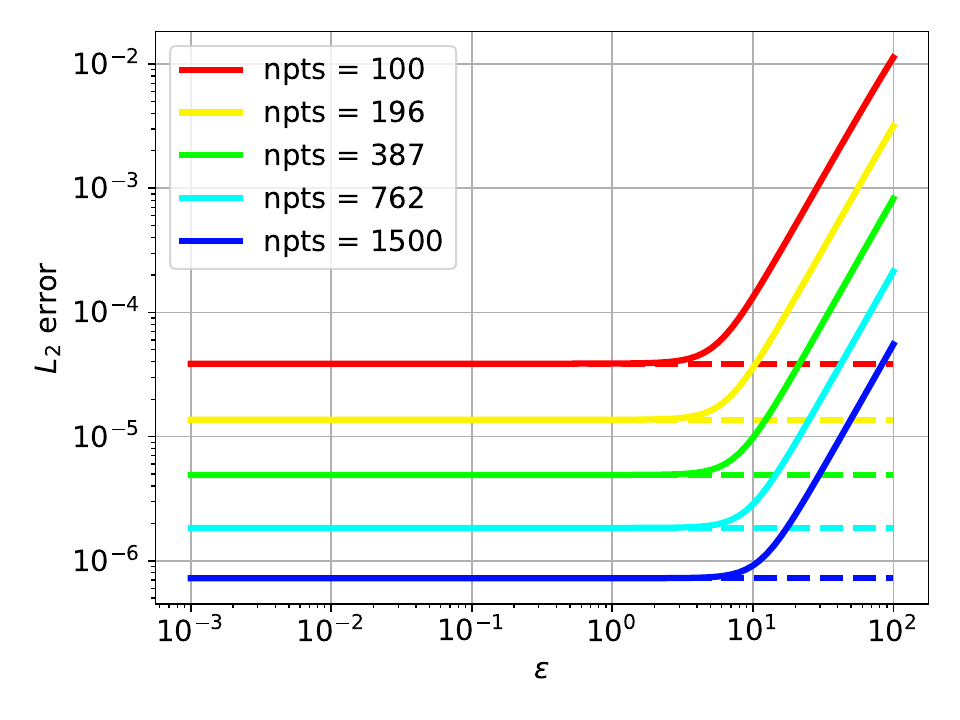}&
\includegraphics[width=.5\textwidth]{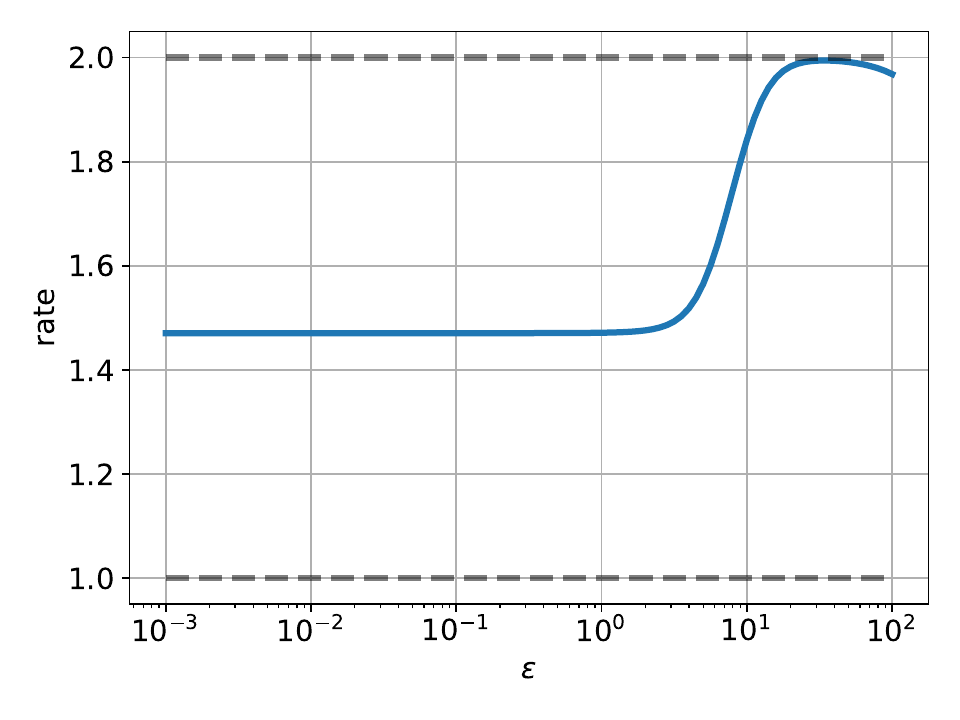}\\
\includegraphics[width=.5\textwidth]{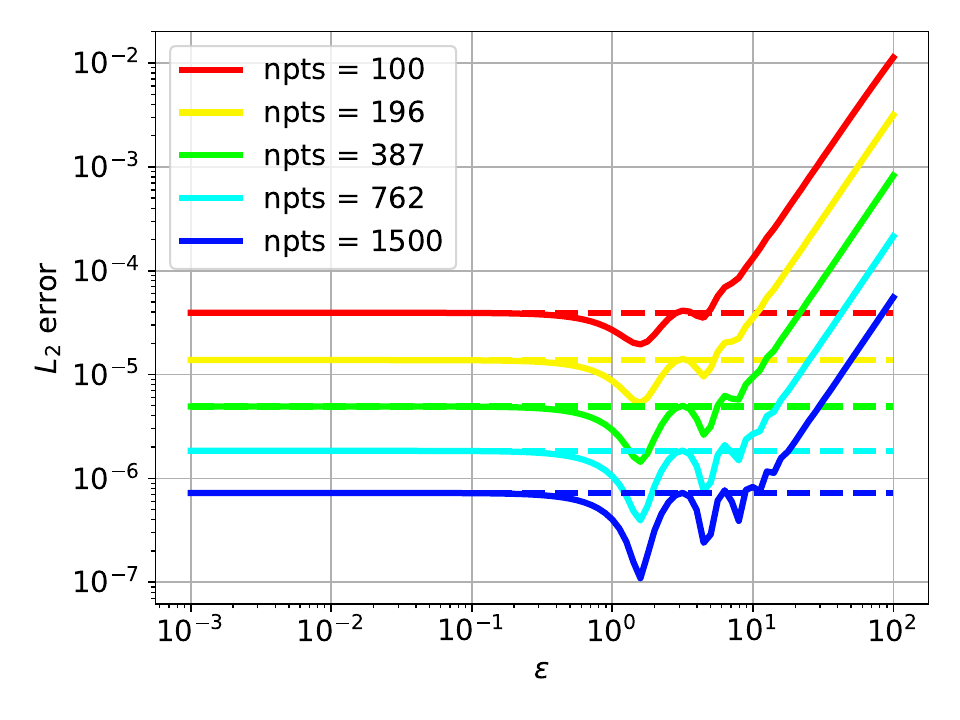}&
\includegraphics[width=.5\textwidth]{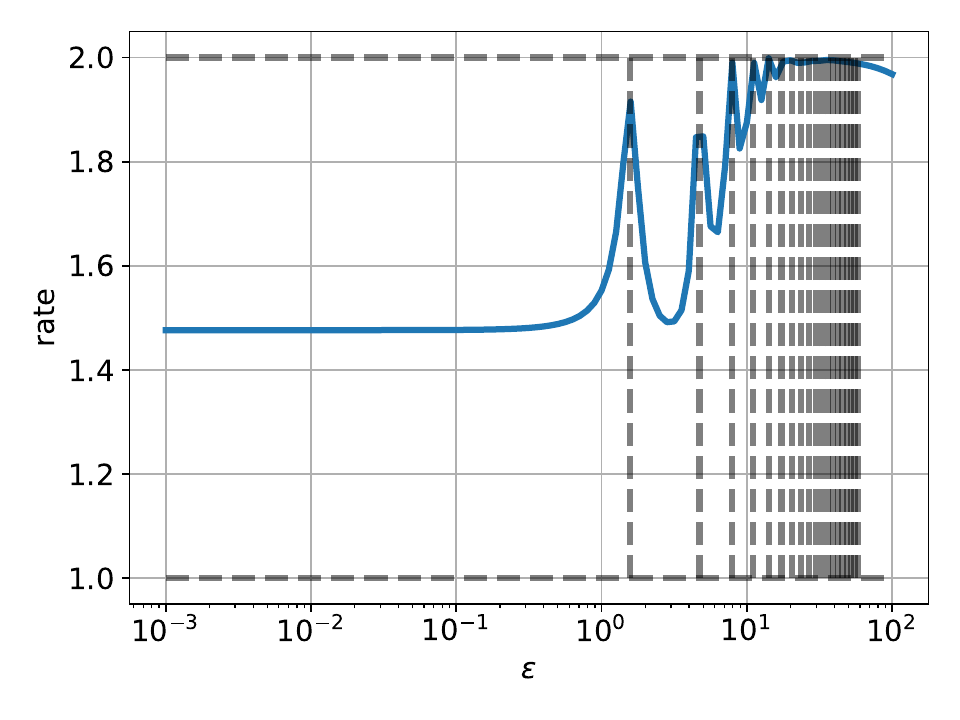}\\
\includegraphics[width=.5\textwidth]{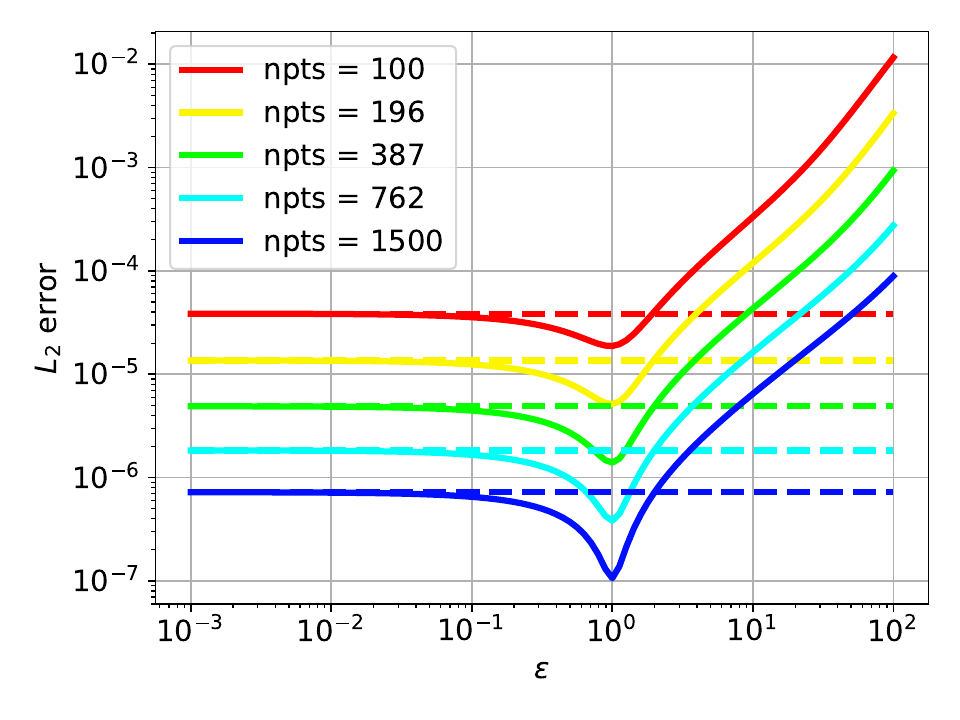}&
\includegraphics[width=.5\textwidth]{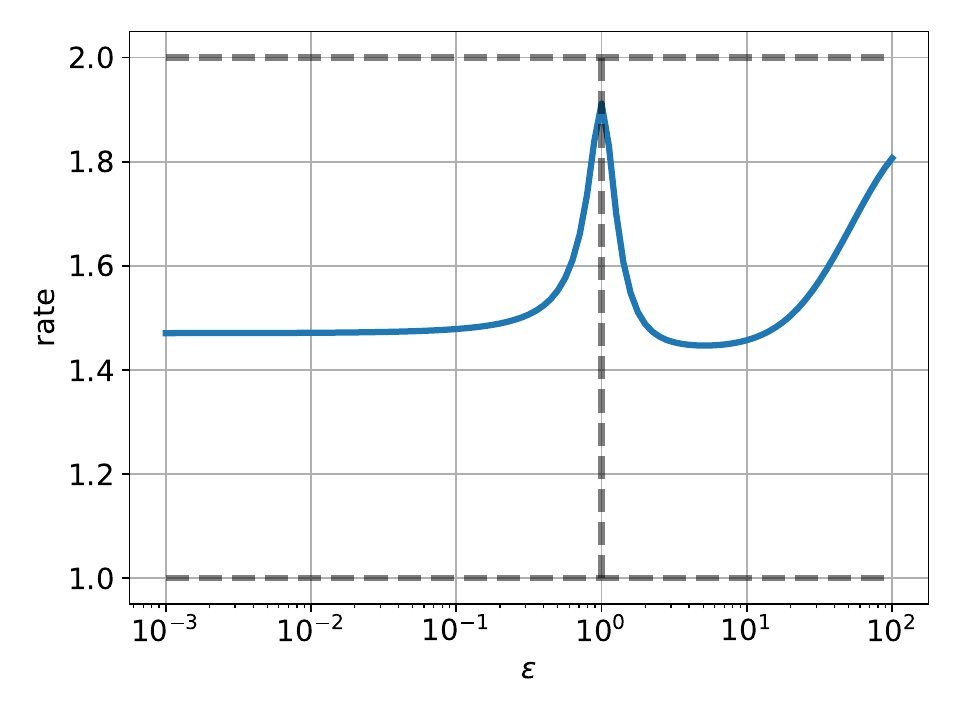}\\
\end{tabular}
\caption{Example~\ref{ex:dip}: Error decay (first column) and rate of decay (second column) as a function of $\varepsilon$ for increasing $n$, and for the 
three kernels. 
The first column additionally reports the interpolation error in the limit $\varepsilon\to0$ extended to all 
$\varepsilon$ (dotted lines).
The second column shows horizontal lines for the basic rate in $\ns$ and the full superconvergence rate, and vertical lines for the expected optimal shape 
parameter(s), when existing.}\label{fig:ex2_1d}
\end{figure}
In this case the error rates (right column) show the single and multiple predicted optimal values for $k^2, k^3$ (vertical dashed lines), while (almost)
optimal rates are obtained for large $\varepsilon$ for all kernels, due to the localization of the kernels and the essentially vanishing boundary conditions.
Note that \Cref{fig:intro_placeholder} displays the red line of the middle left plot of \Cref{fig:ex2_1d},
and additionaly further error curves for more points.
\end{example}

\subsection{Preasymptotic optimal shape parameter}
\label{subsec:preasymp_optimal_shape}

The previous three subsections focused on the asymptotically optimal shape parameter.
In this subsection we continue the discussion raised in \Cref{subsec:non_asymp_vs_asymp}:
We point out that there may be preasymptotic optimal shape parameters, 
which are however not asymptotically optimal in the sense of \Cref{def:optimal_shape}.
This phenomenon thus highlights again
that any algorithm that operates on a finite number of data points
cannot be able to detect an asymptotically optimal shape parameter.

We choose $\Omega = [0, 1]^2 \subset \R^2$ and the linear Matérn kernel,
given by the function $\Phi(x) = (1+\Vert x \Vert) \exp(-\Vert x \Vert)$.
Then we consider the top Mercer eigenfunction $\varphi_{j=1, \varepsilon=1}$,
computed numerically via the top eigenvector of a large kernel matrix based on a grid of $100 \times 100$ equidistant points \cite{santin2016approximation}.
As target functions (given by the discretization on that grid), we consider
\begin{align*}
f_1(x) &:= \varphi_{j=1, \varepsilon=1}(x), \\
f_2(x) &:= \varphi_{j=1, \varepsilon=1} + \gamma | x^{(1)} - 0.5|
\end{align*}
for $\gamma \ll 1$.
The first function $f_1$ is in the superconvergence regime, 
while $f_2$ is likely not due to the additional perturbation.
Next, we consider the approximation of these functions $f_1, f_2$ using up to 600 $P$-greedy (thus well distributed \cite{wenzel2021novel}) points (out of the $100^2$ points), using various shape parameter $\varepsilon$ between $10^{-2}$ and $10^1$.
The resulting $L_2(\Omega)$ approximation errors (computed on the grid of $100^2$ points) are displayed in \Cref{fig:conv_shape_para_2D}.

In the left plot of \Cref{fig:conv_shape_para_2D}, 
which shows the approximation of $f_1$,
one can see a similar behaviour as in \Cref{subsec:no_unique_multiple}.
In particular,
one can clearly see the dip of the $L_2(\Omega)$-error at the shape parameter $\varepsilon = 1$,
which is due to the choice of $f_1$ as $\varphi_{j=1, \varepsilon=1}$.
In the right plot of \Cref{fig:conv_shape_para_2D},
which shows tha approximation of $f_2$,
one can only see a dip of the $L_2(\Omega)$-error at $\varepsilon = 1$ for a small number of interpolation points (e.g.\ $|X| = 10$),
though no longer for a larger number of interpolation points (e.g.\ for $|X| = 600$).
This can be explained as follows:
The function $f_2$ consists of the two parts $\varphi_{j=1, \varepsilon=1}$, which allows for superconvergence and thus can be approximated quickly,
and the perturbation part $\gamma |x^{(1)} - 0.5|$.
For few interpolation points (e.g.\ $|X| = 10$),
the approximation error for both parts has small magnitude --
due to $\varphi_{j=1, \varepsilon=1}$ allowing for superconvergence and $\gamma \ll 1$ being small.
However, increasing the number of interpolations points (e.g.\ $|X| = 600$) the approximation error of $| x^{(1)} - 0.5 |$ becomes relatively smaller than $\gamma$,
thus the overall error of $\gamma |x^{(1)} - 0.5 |$ is larger than that for $\varphi_{j=1, \varepsilon=1}$,
causing the dip to start vanishing.

Thus this example highlights that a simple weighted superposition of a superconvergent part and a non-superconvergent part can lead to a situation where a seemingly optimal shape parameter (for $|X|=10$)
turns out to be no longer optimal in the asymptotic regime (e.g.\ for $|X| = 600$).

We remark that similar examples can especially also be built in the univariate setting,
where the corresponding power spaces and occuring boundary conditions can be determined more explicitly.
In order to highlight that the analysis and implications also hold beyond the univariate case,
a two dimensional example was shown,
which can even be generalized further:
Without any doubt,
similar examples can be done for more general Lipschitz domains $\Omega$ than the square.
Furthermore,
the same behaviour can be observed when replacing $\varphi_{j=1, \varepsilon=1}$ by a more general linear combination of some top eigenfunctions, 
e.g.\ $\{ \varphi_{j, \varepsilon=1} \}_{j=1, ..., 5}$.
Finally, the well distributed $P$-greedy points can also be replaced by a randomly selected subset (thus not having quasi-uniformity anymore) -- the resulting plots looks qualitatively the same.

\begin{figure}[t]
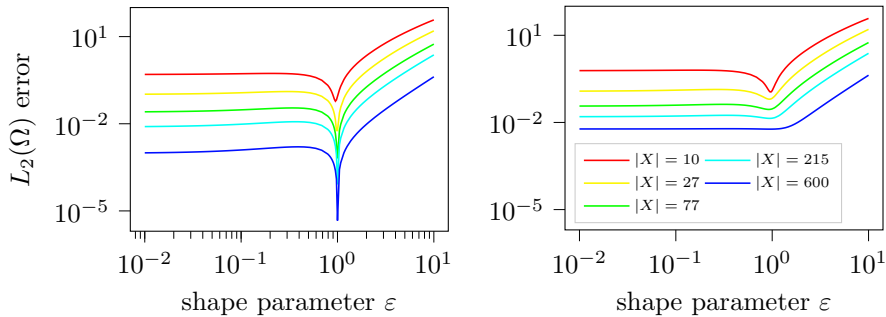

\centering
\setlength\fwidth{.5\textwidth}
\input{Figures/vis_optimal_shape_para_2D.tex} 
\input{Figures/vis_optimal_shape_para_2D_perturbed.tex} 
\caption{Visualization of the error $\Vert f - s_{f, X} \Vert_{L_2(\Omega)}$ over the shape parameter $\varepsilon$ for $\Omega = [0, 1]^2 \subset \R^2$ using various sizes of $X$.
Left: $f_1 := \varphi_{j=1, \varepsilon=1}$, i.e.\ we obtain full superconvergence when using the shape parameter $\varepsilon=1$.
Right: $f_2 := \varphi_{j=1, \varepsilon=1} + \gamma |\cdot_1-0.5|$ for $\gamma$ small.
Due to the additional non-smooth perturbation (which does not allow for superconvergence),
the dip at $\varepsilon=1$ vanishes when using many points (and thus having small errors).}
\label{fig:conv_shape_para_2D}
\end{figure}

\section{Conclusion and Outlook}
\label{sec:conclusion_outlook}

In this article we proposed a new point of view on (shape) parameter optimization for (RBF) kernel interpolation. 
Instead of optimizing these parameters for approximation on a single, finite point set $X \subset \Omega$, 
we look instead at the rate of asymptotic decay of the error on 
sequences of uniform points. 
With this perspective, and relying on recent results on sharp direct and inverse statements providing a one-to-one connection between smoothness and convergence for Sobolev kernels, we were able to 
show that there are instances where the correct choice of the parameters may lead to faster rates of approximation for a given function.
We then discussed the details, potentials and limitations of these theoretical insights, 
presenting favorable and negative examples, and highlighting a number 
of phenomena that were not discussed before in the literature: 
The possible existence of no or multiple shape parameters, and the lack of 
asymptotically optimal parameters for rough functions.

This point of view on parameter optimization could naturally lead to research in further directions. 
First, we considered the most standard case of $L_2(\Omega)$ errors for interpolation on uniform points with Sobolev kernels, but each of these assumptions could 
possibly be relaxed to consider more general settings, as outlined in~\Cref{subsec:disussion}.
Moreover, the role of the shape parameter as a kernel-optimizer could further be extended by addressing non-RBF kernels that are data- or function-optimized, 
such as in~\Cref{ex:ex2_2L_kernels} or even for VSKs.
As mentioned in~\Cref{subsec:literature_optimal_shape_para}, the flat limit is a natural direction to look at when optimizing the shape parameter $\varepsilon$ in RBF kernels. 
Although we pointed out some connections in~\Cref{subsec:no_unique_multiple}, further investigation seems necessary to understand e.g. the role of boundary 
conditions and of exact reproducibility of certain (polynomial) functions.

Finally, although our article has a theoretical viewpoint, we expect that these results will help to understand existing practical algorithms and deriving new ones
for obtaining good parameters or even optimized kernels for %
applications.

\bibliography{references}				%
\bibliographystyle{abbrv}

\appendix

\section{Additional proofs}
\label{sec:additional_proofs}

\begin{proof}[Details on proof of \Cref{th:result_direct}:]

To extend the result to any Sobolev kernel,
we let $k'$ be any translational-invariant reproducing kernel of $H^\tau(\Omega)$ and lef $s_{f, k', X}$ denote the corresponding interpolant.
Letting $\beta := \vartheta \theta$ 
we have $f \in (\ns)_\vartheta \asymp H^{\vartheta \tau}(\Omega)$,
and we can use Theorem 4.2 in \cite{narcowich2006sobolev}
to get
\begin{align*}
\Vert f - s_{f, k, X} \Vert_{L_2(\Omega)} 
&\leq \Vert f - s_{f, k, X}' \Vert_{L_2(\Omega)} + \Vert s_{f, k, X}' - s_{f, k, X} \Vert_{L_2(\Omega)}\\
&\leq C_f h_X^{\vartheta \tau} + \Vert s_{f, k, X}' - s_{f, k, X} \Vert_{L_2(\Omega)}
\end{align*}
where $C>0$ depends on $X$ only via the bound $\rho_0$ on the uniformity constant.
For the second term we use the fact that $s_{f, k, X}$ is the $k$-interpolant of $s_{f, k', X}$ on $X$,
and $s_{f, k', X} \in H^\tau(\Omega)$.
We then apply Theorem 2.12 in \cite{narcowich2005sobolev}, giving
\begin{align*}
\Vert s_{f, k', X} - s_{f, k, X} \Vert_{L_2(\Omega)} 
&= \Vert s_{f, k', X} - s_{s_{f, k', X}, k, X} \Vert_{L_2(\Omega)} \\
&\leq C' h_X^\tau \Vert s_{f, k', X} \Vert_{\ns} \\ 
&\leq C'' h_X^\tau \Vert s_{f, k', X} \Vert_{\mathcal{H}_{k'}(\Omega)},
\end{align*}
where the second inequality uses the norm-equivalences between $\ns, H^\tau(\Omega)$ and $\mathcal{H}_{k'}(\Omega)$.
Finally,
since $f \in H^{\vartheta \tau}(\Omega)$,
by Theorem 3.3 in \cite{avesani2025sobolev}\footnote{
Note that \cite[Theorem 3.3]{avesani2025sobolev} is originally formulated for a sequence of quasi-uniform points and RBF kernels.
However this can easily extended,
as e.g.\ \cite{wenzel2026sharp} extends the necessary Bernstein inequality to non-radial kernels.}
we
can bound the growth of $\Vert s_{f, k', X} \Vert_{\mathcal{H}_{k'}(\Omega)}$,
i.e.\ there is $C''' > 0$ such that
\begin{align*}
\Vert s_{f, k', X} \Vert_{\mathcal{H}_{k'}(\Omega)} \leq C''' h_X^{\beta - \tau}.
\end{align*}
All together we obtain
\begin{align*}
\Vert f - s_{f, k, X} \Vert_{L_2(\Omega)} \leq (C \Vert f \Vert_{H^\beta(\Omega)} + C'' C''') h_X^\beta,
\end{align*}
concluding the proof.

\end{proof}

\end{document}